\documentclass[12pt]{amsart}
\usepackage{amssymb}

\newtheorem{theorem}{Theorem}[section]
\newtheorem{lemma}[theorem]{Lemma}

\newtheorem{definition}[theorem]{Definition}

\newtheorem{corollary}[theorem]{Corollary}
\newtheorem{proposition}[theorem]{Proposition}

\newtheorem{lem-def}[theorem]{Lemma-Definition}

\DeclareRobustCommand\longtwoheadrightarrow
{\relbar\joinrel\twoheadrightarrow}

\newcommand{\hooklongrightarrow}{\lhook\joinrel\longrightarrow}

\renewenvironment{proof}{{\bfseries Proof.}}{\qed}
\newcommand{\Z}{\mathbb Z}
\newcommand{\Q}{\mathbb Q}

\def\op{\operatorname}

\def\al{\alpha}
\def\as#1{\renewcommand\arraystretch{#1}}

\def\bs{\vskip.5cm}

\def\diso{\lower.4ex\hbox{$\downarrow$}\raise.4ex\hbox{\mbox{\scriptsize
$\wr$}}}
\def\dm{\Delta}

\def\e{\medskip}

\def\ep#1{\exp(\Pi i#1)}
\def\ep{\epsilon}

\def\g{\Gamma}
\def\ga{\gamma}

\def\gchi{\g_{\deg(\chi)}}
\def\gen#1{\big\langle\, {#1} \,\big\rangle}

\def\gg{\mathcal{G}}
\def\ggalg{\mathcal{G}_v^{\op{al}}}
\def\ggm{\mathcal{G}_\mu}

\def\gm{\g_\mu}

\def\gp{\mathfrak{p}}
\def\gphi{\g_{\deg(\phi)}}

\def\ggv{\mathcal{G}_v}

\def\gv{\Gamma_v}

\def\hm{H_\mu}

\def\ic{\op{irc}}
\def\imp{\,\Longrightarrow\,}

\def\iso{\ \lower.3ex\hbox{\as{.08}$\begin{array}{c}\lra\\\mbox{\tiny $\sim\,$}\end{array}$}\ }

\def\kal{k^{\op{al}}}

\def\km{k_\mu}

\def\kpm{\op{KP}(\mu)}

\def\kv{k_v}
\def\kx{K[x]}

\def\kxchi{K[x]_{\deg(\chi)}}

\def\lrc{\op{lrc}}

\def\lg{l\raise.6ex\hbox to.2em{\hss.\hss}l}
\def\ll{\mathcal{L}}
\def\lra{\,\longrightarrow\,}
\def\m{{\mathfrak m}}
\def\md#1{\; \mbox{\rm(mod }{#1})}

\def\mmu{\mid_\mu}
\def\mn{\op{Min}}
\def\mua{\mu_\al}
\def\mx{\op{Max}}

\def\nlc{\op{nlc}}

\def\nmu{\nmid_\mu}

\def\oo{\mathcal{O}}
\def\orb{\hbox to  .3em{$\backslash$}\backslash}
\def\ord{\op{ord}}

\def\ppa{\mathcal{P}_{\alpha}}
\def\pset{\mathcal{P}}

\def\qg{\mathbb{Q}\g}

\def\qgm{\mathbb{Q}\gm}

\def\rr{\mathcal{R}}
\def\rrm{\mathcal{R}_\mu}

\def\sii{\,\Longleftrightarrow\,}

\def\smu{\sim_\mu}

\def\t{\theta}

\def\Min{\mathrm{Min}}

\newcounter{cs}
\stepcounter{cs}
\newcommand{\casos}{\begin{itemize}}
\newcommand{\fcasos}{\end{itemize}\setcounter{cs}{1}}

\newfont{\tit}{cmr12 scaled \magstep3}

\title{Key polynomials over valued fields}

\makeatletter
\@namedef{subjclassname@2010}{%
  \textup{2010} Mathematics Subject Classification}

\author{Enric Nart}
\address{Departament de Matem\`{a}tiques,
         Universitat Aut\`{o}noma de Barcelona,
         Edifici C, E-08193 Bellaterra, Barcelona, Catalonia, Spain}
\email{nart@mat.uab.cat}
\thanks{Partially supported by grant MTM2016-75980-P from the Spanish MEC}
\date{}

\begin{document}

\maketitle

\section*{Introduction}
Key polynomials over a valued field $(K,v)$ were introduced by S. MacLane as a tool to construct augmentations of discrete rank-one valuations on the polynomial ring $\kx$ \cite{mcla}. As an application, MacLane designed an algorithm to compute all extensions of the given valuation $v$ on $K$ to a finite field extension $L/K$ \cite{mclb}.

This work was generalized to arbitrary valuations by M. Vaqui\'e \cite{Vaq} and, independently, by F.J. Herrera, M.A. Olalla and M. Spivakovsky \cite{hos}. 

In the non-discrete case, \emph{limit augmented valuations} arise. The structure of their graded algebra, and the description of their sets of key polynomials are crucial questions, linked with the study of the defect of a valuation in a finite extension, and the local uniformization problem \cite{hmos,mahboub,SS,Vaq2}.      

In this paper, we fix an arbitrary valuation $\mu$ on $\kx$, and we determine the structure of its graded algebra $\ggm$, and describe its set of key polynomials $\kpm$, in terms of a key polynomial of minimal degree. We also characterize valuations not admitting key polynomials (Theorem \ref{kpempty}).

Some of the results of the paper can be found in \cite{Vaq}, but only for augmented valuations. Also, in \cite{PP} some partial results are obtained for residually transcendental valuations, by using the fact that these valuations are determined by a minimal pair. 

In our approach, we do not make any assumption on $\mu$, and we derive our results in a pure abstract form, from the mere existence of key polynomials.

In section 2 we study general properties of key polynomials, while in section 3, we study specific properties of key polynomials of minimal degree. In section 4, we determine the structure of the graded algebra. 

Section 5 is devoted to the introduction of residual polynomial operators, based on old ideas of Ore and MacLane \cite{ore,mcla}. These operators yield a malleable and elegant tool, able to replace the onerous ``lifting" techniques in the context of valuations constructed from minimal pairs.

In section 6 we describe the set of key polynomials and we prove that a certain \emph{residual ideal operator} sets a bijection
$$
\kpm/\!\smu\ \lra\ \mx(\Delta)
$$
between the set of $\mu$-equivalence classes of key polynomials, and the maximal spectrum of the subring $\Delta\subset \ggm$, piece of degree zero in the graded algebra. This result is inspired in \cite{ResidualIdeals}, where it was proved for discrete rank-one valuations.

Finally, in section 7, we single out a key polynomial of minimal degree for augmented and limit augmented valuations. In this way, the structure of the graded algebra and the set of key polynomials for these valuations can be obtained from the results of this paper.

\section{Graded algebra of a valuation on a polynomial ring}
\subsection{Graded algebra of a valuation}
Let $\g$ be an ordered abelian group. Consider
$$
w\colon L\lra \g\cup\{\infty\}
$$
a valuation on a field $L$, and denote\e

\begin{itemize}
\item $\m_w\subset \oo_w\subset L$, the maximal ideal and valuation ring of $w$. \e
\item $k_w=\oo_w/\m_w$, the residue class field of $w$.\e
\item $\g_w=w(L^*)$, the group of values of $w$.
\end{itemize}\e

To any subring $A\subset L$ we may associate a graded algebra as follows.

For every $\alpha \in \g_w$, consider the $\oo_w$-submodules:
$$
\ppa=\{a\in A\mid w(a)\ge \alpha\}\supset
\ppa^+=\{a\in A\mid w(a)> \alpha\},
$$ 
leading to the graded algebra
$$
\operatorname{gr}_w(A)=\bigoplus\nolimits_{\alpha\in\g_w}\ppa/\ppa^+.
$$
The product of homogeneous elements is defined in an obvious way:
$$
\left(a+\ppa^+\right)\left(b+\mathcal{P}_{\beta}^+\right)=ab+\mathcal{P}_{\alpha+\beta}^+.
$$

If the classes  $a+\ppa^+$, $b+\mathcal{P}_{\beta}^+$ are different from zero, then $w(a)=\alpha$, $w(b)=\beta$. Hence, $w(ab)=\alpha+\beta$, so that $ab+\mathcal{P}_{\alpha+\beta}^+$ is different from zero too. 

Thus, $\operatorname{gr}_w(A)$ is an integral domain.\e

Consider the ``initial term" mapping \;$H_w\colon A\to \operatorname{gr}_w(A)$,\; given by
$$
H_w(0)=0,\qquad H_w(a)=a+\pset^+_{w(a)},\ \mbox{ for }a\in A,\ a\ne0.
$$

Note that $H_w(a)\ne0$ if $a\ne0$. For all $a,b\in A$ we have:
\begin{equation}\label{Hmu}
\as{1.3}
\begin{array}{l}
 H_w(ab)=H_w(a)H_w(b), \\
 H_w(a+b)=H_w(a)+H_w(b), \ \mbox{ if }w(a)=w(b)=w(a+b).
\end{array}
\end{equation}

\begin{definition}
Two elements $a,b\in A$ are said to be \emph{$w$-equivalent} if 
$H_w(a)=H_w(b)$.

In this case,  we write $a\sim_w b$.
This is equivalent to $w(a-b)>w(b)$. \e

We say that $a$ is \emph{$w$-divisible} by $b$ if $H_w(a)$ is divisible by $H_w(b)$ in $\operatorname{gr}_w(A)$. 

In this case,  we write $b\mid_w a$.
This is equivalent to $a\sim_w bc$, \ for some $c\in A$.\e
\end{definition}

\subsection{Valuations on polynomial rings. General setting}

Throughout the paper, we fix a field $K$ and a valuation 
$$\mu\colon K(x) \lra \gm \cup \{\infty\},$$
on the field $K(x)$ of rational functions in one indeterminate $x$. 

We do not make any assumption on the rank of $\mu$.\e

We denote by  $v=\mu_{|_K}$ the valuation on $K$ obtained by the restriction of $\mu$. The group of values of $v$ is a subgroup of $\gm$:
$$
\gv=v(K^*)=\mu(K^*)\subset \gm.
$$


For each one of the two valuations $v$, $\mu$, we consider a different graded algebra:
$$
\ggv:=\operatorname{gr}_v(K),\qquad \ggm:=\operatorname{gr}_\mu(\kx).
$$

In the algebra $\ggv$, every non-zero homogeneous element is a unit: $$H_v(a)^{-1}=H_v(a^{-1}),\qquad \forall a\in K^*.$$

The subring of homogeneous elements of degree zero of $\ggv$ is $\kv$, so that $\ggv$ has a natural structure of $\kv$-algebra.\e

We have a natural embedding of graded algebras,
$$
\ggv\hookrightarrow \ggm,\qquad a+\pset^+_\alpha(v)\,\mapsto\, a+\pset^+_\alpha(\mu),\quad \forall\alpha\in\gv,\ \forall a\in \pset_\alpha(v). 
$$

The subring of $\ggm$ determined by the piece of degree zero is denoted 
$$
\dm=\dm_\mu=\pset_0(\mu)/\pset_0^+(\mu).
$$
Since $\oo_v\subset\pset_0=\kx\cap \oo_\mu$, and $\m_v=\pset_0^+\cap \oo_v\subset \pset_0^+=\kx\cap \m_\mu$, there are canonical injective ring homomorphisms: 
$$\kv\hooklongrightarrow\dm\hooklongrightarrow \km.$$
In particular, $\dm$ and $\ggm$ are equipped with a canonical structure of $\kv$-algebra. \e

The aim of the paper is to analyze the structure of the graded algebra $\ggm$ and show that most of the properties of the extension $\mu/v$ are reflected in algebraic properties of the extension $\ggm/\ggv$.

For instance, an essential role is played by the \emph{residual ideal operator}
\begin{equation}\label{resIdeal}
\rr=\rrm\colon \kx\lra I(\dm),\qquad g\mapsto \left(\hm(g)\ggm\right)\cap \dm,
\end{equation}
where $I(\dm)$ is the set of ideals in $\dm$.

In sections \ref{secR} and \ref{secKPuf}, we shall study in more detail this operator $\rr$, which translates questions about the action of  $\mu$ on $\kx$ into ideal-theoretic problems in the ring $\dm$. 

\subsection*{Commensurability}
The \emph{divisible hull} of an ordered abelian group $\g$ is
$$
\qg:=\g\otimes_\Z\Q.
$$
This $\Q$-vector space inherits a natural structure of ordered abelian group, with the same rank as $\g$.\e

The \emph{rational rank} of $\g$ is defined as $\op{rr}(\g)=\dim_\Q(\qg)$.\e

Since $\g$ has no torsion, it admits an order-preserving embedding $\g\hookrightarrow\qg$ into its divisible hull. For every $\ga\in\qg$ there exists a minimal positive integer $e$ such that $e\ga\in\g$.\e

We say that our extension $\mu/v$ is \emph{commensurable} if $\qg_v=\qg_\mu$, or in other words, if $\op{rr}(\gm/\gv)=0$. This is equivalent to $\gm/\gv$ being a torsion group.\e

Actually, $\op{rr}(\gm/\gv)$ takes only the values $0$ or $1$, as the following well-known inequality shows \cite[Thm. 3.4.3]{valuedfield}:
\begin{equation}\label{ftalineq}
\op{tr.deg}(\km/\kv)+\op{rr}(\gm/\gv)\le \op{tr.deg}(K(x)/K)=1. 
\end{equation}

Finally, we fix some notation to be used throughout the paper.\e

\noindent{\bf Notation. }For any positive integer $m$ we denote
$$
\kx_m=\{a\in\kx\mid \deg(a)<m\}.
$$

For any polynomials $f,\chi\in\kx$, with $\deg(\chi)>0$, we denote the canonical $\chi$-expansion of $f$ by: 
$$
f=\sum\nolimits_{0\le s}f_s\chi^s, 
$$
being implicitly assumed that the coefficients $f_s\in \kx$ have $\deg(f_s)<\deg(\chi)$.

\section{Key polynomials. Generic properties}\label{secKP}
In this section, we introduce the concept of key polynomial for $\mu$, and we study some generic properties of key polynomials. 

In section \ref{secKPmindeg}, we shall see that, if $\mu$ admits key polynomials at all, then the structure of $\ggm$ is determined by any key polynomial of minimal degree.

\begin{definition}\label{mu}Let $\chi\in \kx$.



We say that $\chi$ is $\mu$-irreducible if $\hm(\chi)\ggm$ is a non-zero prime ideal. 

We say that $\chi$ is $\mu$-minimal if $\chi\nmid_\mu f$ for any non-zero $f\in \kx$ with $\deg f<\deg \chi$.
\end{definition}


The property of $\mu$-minimality admits a relevant characterization, given in Proposition \ref{minimal0} below.

\begin{lemma}\label{minimal1}
Let $f,\chi\in \kx$. Consider a $\chi$-expansion of $f\in \kx$ as follows:
$$f=\sum\nolimits_{0\le s}a_s\chi^s, \qquad  a_s\in \kx,\quad \chi\nmid_\mu a_s,\ \forall s.$$ 
Then, $\mu(f)=\Min\{\mu(a_s\chi^s)\mid 0\le s\}$. 
\end{lemma}

\begin{proof}
Write $f=a_0+\chi q$ with $q\in \kx$. Then, $\mu(f)\geq\Min\{\mu(a_0),\mu(\chi q)\}$.

A strict inequality would imply $a_0\smu-\chi q$, against our assumption. Hence,  equality holds, and the result follows from a recurrent argument.
\end{proof}\bs

\begin{proposition}\label{minimal0}
Let $\chi\in \kx$ be a non-constant polynomial.  
The following conditions are equivalent.
\begin{enumerate}
\item $\chi$ is $\mu$-minimal
\item For any $f\in \kx$, with $\chi$-expansion $f=\sum\nolimits_{0\le s}f_s\chi^s$, we have 
$$\mu(f)=\Min\{\mu\left(f_s\chi^s\right)\mid 0\le s\}.$$ 
\item For any non-zero $f\in \kx$, with $\chi$-expansion $f=\sum\nolimits_{0\le s}f_s\chi^s$, we have
$$\chi\nmid_\mu f\sii\mu(f)=\mu(f_0).$$
\end{enumerate}
\end{proposition}

\begin{proof}  
The implication (1)$\imp$(2) follows from Lemma \ref{minimal1}. In fact, if $\chi$ is $\mu$-minimal, then $\chi\nmid_\mu f_s$ for all $s$, because $\deg(f_s)<\deg(\chi)$.\e

Let us deduce (3) from (2). Take a non-zero $f\in \kx$ and write $f=f_0+\chi q$ with $q\in \kx$. By item (2), we have $\mu(f) \le \mu(f_0)$.

If $\mu(f)<\mu(f_0)$, then $f\smu \chi q$, so $\chi\mmu f$. Conversely, if $f \smu \chi g$ for some $g\in \kx$, 
then $\mu(f-\chi g) > \mu(f)$. Since the $\chi$-expansion of $f-\chi g$ has the same $0$-th coefficient $f_0$, condition $(2)$ shows that $\mu(f) <\mu(f-\chi g)\leq\mu(f_0)$. \e

Finally, (3) implies (1). If $\deg(f)<\deg(\chi)$, then the $\chi$-expansion of $f$ is $f=f_0$. By item $(3)$, $\chi \nmid_{\mu}f$. 
\end{proof}\bs

The property of $\mu$-minimality is not stable under $\mu$-equivalence. For instance, if $\chi$ is $\mu$-minimal and $\mu(\chi)>0$, then $\chi+\chi^2\smu \chi$ and $\chi+\chi^2$ is not $\mu$-minimal. However, for $\mu$-equivalent polynomials of the same degree, $\mu$-minimality is clearly preserved. 

\begin{definition} A \emph{key polynomial} for $\mu$ is a monic polynomial in $\kx$ which is $\mu$-minimal and $\mu$-irreducible. 

The set of key polynomials for $\mu$ will be denoted by $\kpm$.
\end{definition}

\begin{lemma}\label{mid=sim}
Let $\chi\in\kpm$, and let $f\in \kx$ a monic polynomial such that $\chi\mmu f$ and $\deg(f)=\deg(\chi)$. Then, $\chi\smu f$ and $f$ is a key polynomial for $\mu$ too.
\end{lemma}

\begin{proof}
The $\chi$-expansion of $f$ is of the form $f=f_0+\chi$, with $\deg(f_0)<\deg(\chi)$. Items (2) and (3) of Proposition \ref{minimal0} show that $\mu(f)<\mu(f_0)$. Hence, $\hm(f)=\hm(\chi)$, and $f$ is $\mu$-irreducible. Since $\deg(f)=\deg(\chi)$, it is $\mu$-minimal too.
\end{proof}\e

\begin{lemma}\label{ab}
Let $\chi\in\kpm$.
\begin{enumerate}
\item For $a,b\in \kxchi$, let $ab=c+d\chi$ be the $\chi$-expansion of $ab$. Then,
$$\mu(ab)=\mu(c)\le\mu(d\chi).$$
\item $\chi$ is  irreducible in $\kx$.
\end{enumerate}
\end{lemma}

\begin{proof}
For any $a,b\in \kxchi$, we have $\chi\nmid_\mu a$,  $\chi\nmid_\mu b$ by the $\mu$-minimality of $\chi$. Hence,  $\chi\nmid_\mu ab$ by the $\mu$-irreducibility of $\chi$. Thus, (1) follows from Proposition \ref{minimal0}.

In particular, the equality $\chi=ab$ is impossible, so that $\chi$ is irreducible.
\end{proof}

\subsection*{Minimal expression of $\hm(f)$ in terms of $\chi$-expansions}

\begin{definition}\label{order}
For $\chi\in\kpm$ and a non-zero $f\in \kx$, we let $s_\chi(f)$ be the largest integer $s$ such that $\chi^s\mid_\mu f$. 

Namely, $s_\chi(f)$ is the order with which the prime $\hm(\chi)$ divides $\hm(f)$ in $\ggm$. 

Accordingly, by setting \,$s_\chi(0):=\infty$, we get
\begin{equation}\label{multiplicative}
s_\chi(fg)=s_\chi(f)+s_\chi(g),\ \mbox{ for all }f,g\in \kx.
\end{equation}
\end{definition}

\begin{lemma}\label{sphi}
Let $f\in \kx$ with $\chi$-expansion $f=\sum_{0\le s}f_s\chi^s$. Denote
$$
I_\chi(f)=\left\{s\in\Z_{\ge0}\mid \mu(f_s\chi^s)=\mu(f)\right\}.
$$
Then, $f\smu \sum\nolimits_{s\in I_\chi(f)}f_s\chi^s$,
and\, $s_\chi(f)=\Min(I_\chi(f))$.
\end{lemma}

\begin{proof}
Let $g=\sum\nolimits_{s\in I_\chi(f)}f_s\chi^s$.
By construction, $f-g=\sum_{s\not\in I}f_s\chi^s$ has $\mu$-value $\mu(f-g)>\mu(f)$. This proves $f\smu g$. In particular, $s_\chi(f)=s_\chi(g)$. 

If $s_0=\Min(I_\chi(f))$, we may write,
$$
g=\chi^{s_0}\left(f_{s_0}+\chi h\right),
$$
for some $h\in \kx$. By construction, $\mu(f_{s_0})=\mu(f_{s_0}+\chi h)=\mu(g/\chi^{s_0})$.

By item (3) of Proposition \ref{minimal0}, $\chi\nmid_\mu\left(f_{s_0}+\chi h\right)$. Therefore,  
$s_\chi(g)=s_0$.
\end{proof}\e

\begin{definition}
For any $f\in\kx$ we denote $s'_\chi(f)=\mx(I_\chi(f))$.  

Denote for simpliciy $s=s_\chi(f)$, $s'=s'_\chi(f)$. The homogeneous elements 
$$\ic(f):=\hm(f_s),\qquad \lrc(f):=\hm(f_{s'}) 
$$
are the \emph{initial residual coefficient} and \emph{leading residual coefficient} of $f$, respectively.
\end{definition}\e

The next lemma shows that $s'_\chi(f)$ is an invariant of the $\mu$-equivalence class of $f$.

\begin{lemma}\label{sprime}
If $f,g\in \kx$ satisfy $f\smu g$, then $I_\chi(f)=I_\chi(g)$, and $f_s\smu g_s$ for all $s\in I_\chi(f)$.
In particular, $\ic(f)=\ic(g)$ and $\lrc(f)=\lrc(g)$.
\end{lemma}

\begin{proof}
Consider the $\chi$-expansions $f=\sum_{0\le s}f_s\chi^s$, $g=\sum_{0\le s}g_s\chi^s$. 

If $f\smu g$, then for any $s\ge0$ we have
\begin{equation}\label{smus}
\mu(f)<\mu(f-g)\le \mu\left((f_s-g_s)\chi^s\right).
\end{equation}

The condition $s\in I_\chi(f)$, $s\not\in I_\chi(g)$ (or viceversa) contradicts (\ref{smus}). In fact, $$\mu\left((f_s-g_s)\chi^s\right)=\mu(f),$$
because $\mu\left(f_s\chi^s\right)=\mu(f)$ and $\mu\left(g_s\chi^s\right)>\mu(g)=\mu(f)$.

Also, for all $s\in I_\chi(f)$, we have $\mu\left(f_s\chi^s\right)=\mu(f)$ and (\ref{smus}) shows that $f_s\chi^s\smu g_s\chi^s$. Thus, $f_s\smu g_s$.
\end{proof}\bs

We shall see in section \ref{secKPmindeg} that the  equality 
$$
s'_\chi(fg)= s'_\chi(f)+s'_\chi(g),\ \mbox{ for all }f,g\in \kx
$$
holds if $\chi$ is a key polynomial of minimal degree. 

\subsection*{Semivaluation attached to a key polynomial}
\begin{lemma}\label{subgroup}
Let $\chi\in\kpm$. Consider the subset $\g\subset\gchi\subset\gm$ defined as
$$ \gchi=\left\{\mu(a)\mid a\in \kxchi,\ a\ne0\right\}.$$
Then, $\gchi$ is a subgroup of $\gm$ and $\gen{\gchi,\mu(\chi)}=\gm$.
\end{lemma}

\begin{proof}
Since $\chi$ is $\mu$-minimal, Proposition \ref{minimal0} shows that $\gen{\gchi,\mu(\chi)}=\gm$. 

By Lemma \ref{ab}, $\gchi$ is closed under addition. 

Take $\mu(a)\in\gchi$ for some non-zero $a\in\kxchi$. The polynomials  $a$ and $\chi$ are coprime, because  $\chi$ is irreducible. Hence, they satisfy a B\'ezout identity
\begin{equation}\label{bezoutab}
ab+\chi\, d=1,\qquad \deg(b)<\deg(\chi),\quad \deg(d)<\deg(a)<\deg(\chi).
\end{equation}

Since $ab=1-d\chi$ is the $\chi$-expansion of $ab$, Lemma \ref{ab} shows that $\mu(ab)=\mu(1)=0$. Hence $-\mu(a)=\mu(b)\in\gchi$. This shows that  $\gchi$ is a subgroup of $\gm$.
\end{proof}\bs

Let $\chi\in\kpm$. Consider the prime ideal $\gp=\chi\,\kx$ and the field $K_\chi=\kx/\gp$. 

By the definition of $\gchi$, we get a well-defined onto mapping:
$$
v_\chi\colon K_\chi^*\longtwoheadrightarrow \gchi,\qquad v_\chi(f+\gp)=\mu(f_0),\quad \forall f\in\kx\setminus\gp,
$$
where $f_0\in\kx$ is the common $0$-th coefficient of the $\chi$-expansion of all polynomials in the class $f+\gp$.

\begin{proposition}\label{vphi}
The mapping $v_\chi$ is a valuation on $K_\chi$ extending $v$, with group of values $\gchi$.
\end{proposition}

\begin{proof}
This mapping $v_\chi$ is a group homomorphism by Lemma \ref{ab}. Finally,
$$
v_\chi((f+g)+\gp)=\mu(f_0+g_0)\ge \Min\{\mu(f_0),\mu(g_0)\}=\Min\{v_\chi(f+\gp),v_\chi(g+\gp)\},
$$
because  $(f+g)_0=f_0+g_0$. Hence, $v_\chi$ is a valuation on $K_\chi$.
\end{proof}\bs

Denote the maximal ideal, the valuation ring and the residue class field of $v_\chi$ by: 
$$\m_\chi\subset\oo_\chi\subset K_\chi,\qquad k_\chi=\oo_\chi/\m_\chi.$$ 

Let $\t\in K_\chi=\kx/(\chi)$ be the root of $\chi$ determined by the class of $x$. 

With this notation, we have $K_\chi=K(\t)$, and
$$
v_\chi(f(\t))=\mu(f_0)=v_\chi(f_0(\t)),\qquad \forall f\in\kx.
$$

We abuse of language and denote still by $v_\chi$ the corresponding semivaluation
$$
\kx\longtwoheadrightarrow K_\chi \stackrel{v_\chi}\lra \gchi\cup\{\infty\}
$$
with support $\chi\kx=v_\chi^{-1}(\infty)$.\e

According to the definition given in (\ref{resIdeal}), the residual ideal $\rr(\chi)$ of a key polynomial $\chi$ is a prime ideal in $\dm$. Let us show that it is actually a maximal ideal in $\dm$.

\begin{proposition}\label{maxideal}
If $\chi$ is a key polynomial for $\mu$, then 
$\rr(\chi)$ is the kernel of the onto homomorphism $$\dm\longtwoheadrightarrow k_\chi,\qquad g+\pset^+_0\ \mapsto\ g(\t)+\m_\chi.$$ 
In particular, $\rr(\chi)$ is a maximal ideal in $\dm$.
\end{proposition}

\begin{proof}
By Proposition \ref{minimal0}, if $g\in\pset_0$, we have $v_\chi(g(\t))=\mu(g_0)\ge \mu(g)\ge0$, so that $g(\t)\in \oo_\chi$. Thus, we get a well-defined ring homomorphism $\pset_0\to k_\chi$. 

This mapping is onto, because every element in $k_\chi$ can be represented as $h(\t)+\m_\chi$ for some $h\in \kxchi$ with $v_\chi(h(\t))\ge0$. Since $\mu(h)=v_\chi(h(\t))\ge0$, we see that $h$ belongs to $\pset_0$. 

Finally, if $g\in\pset_0^+$, then $v_\chi(g(\t))\ge \mu(g)>0$; thus, the above homomorphism vanishes on $\pset_0^+$ and it induces an onto mapping $\dm\twoheadrightarrow k_\chi$.

The kernel of this mapping is the set of all elements $\hm(f)$ for $f\in\kx$ satisfying $\mu(f_0)>\mu(f)=0$. By Proposition \ref{minimal0}, this is equivalent to $\mu(f)=0$ and $\chi\mmu f$. In other words, the kernel is $\rr(\chi)$. 
\end{proof}

\section{Key polynomials of minimal degree}\label{secKPmindeg}
In this section, we study special properties of key polynomials of minimal degree. These objects are crucial for the resolution of the two main aims of the paper:

\begin{itemize}
\item Determine the structure of $\ggm$ as a $\kv$-algebra.\e
\item Determine the structure of the quotient set $\kpm/\!\smu$.
\end{itemize}

These tasks will be carried out in sections \ref{secDelta} and \ref{secKPuf}, respectively.\e

Recall the embedding of graded $\kv$-algebras
$$\ggv\hooklongrightarrow\ggm.$$

Let $\xi\in\ggm$ be a non-zero homogeneous element which is algebraic over $\ggv$. 
Then, $\xi$ satisfies an homogeneous equation:
$$
\ep_0+\ep_1\xi+\cdots+\ep_m\xi^m=0,
$$
with $\ep_0\,\dots,\ep_m$ homogeneous elements in $\ggv$ such that $\deg(\ep_i\xi^i)$ is constant for all indices $0\le i\le m$ for which $\ep_i\ne0$.

Since all non-zero homogeneous elements in $\ggv$ are units, we have:
\begin{equation}\label{xi}
\xi \mbox{ algebraic over }\ggv\imp\xi \mbox{ is a unit in }\ggm, \mbox{ and $\xi$ is integral over }\ggv. 
\end{equation} 

\begin{lemma}\label{alg}
Let $\ggv\subset \ggalg\subset\ggm$ be the subalgebra generated by all homogeneous elements in $\ggm$ which are algebraic over $\ggv$.

If an homogeneous element $\xi\in\ggm$ is algebraic over $\ggalg$, then it belongs to  $\ggalg$. 
\end{lemma}

\begin{proof}
Since all non-zero homogeneous elements in $\ggalg$ are units, the element $\xi$ is integral over $\ggalg$. Hence, it is integral over $\ggv$, so that it belongs to $\ggalg$.   
\end{proof}\e

\begin{theorem}\label{phimindeg}
Let $\phi\in\kx$ be a monic polynomial of minimal degree $n$ such that $\hm(\phi)$ is transcendental over $\ggv$. Then, $\phi$ is a key polynomial for $\mu$.

Moreover, for $a,b\in\kx_n$, let the $\phi$-expansion of $ab$ be 
\begin{equation}\label{abc}
ab=c+d\phi,\qquad c,d\in\kx_n.
\end{equation}
Then, $ab\smu c$.
\end{theorem}

\begin{proof}
Let us first show that $\phi$ is $\mu$-minimal. 

According to Proposition \ref{minimal0}, the $\mu$-minimality of $\phi$ is equivalent to 
$$
\mu(f)=\Min\{\mu\left(f_s\phi^s\right)\mid 0\le s\},
$$
for all $f\in \kx$, being $f=\sum_{0\le s}f_s\phi^s$ its canonical $\phi$-expansion. 

For any given $f\in\kx$, let $\delta=\Min\{\mu\left(f_s\phi^s\right)\mid 0\le s\}$, and consider 
$$
I=\{0\le s\mid \mu\left(f_s\phi^s\right)=\delta\},\qquad f_I=\sum_{s\in I}f_s\phi^s.
$$

We have  $\mu(f)\ge\delta$, and the desired equality $\mu(f)=\delta$ is equivalent to $\mu(f_I)=\delta$.

If $\#I=1$, this is obvious. In the case  $\#I>1$, the equality  $\mu(f_I)=\delta$ follows from the transcendence of $\hm(\phi)$ over $\ggv$. In fact,
$$
\mu(f_I)>\delta \sii \sum_{s\in I}\hm(f_s)\hm(\phi)^s=0.
$$
By the minimality of $n=\deg(\phi)$, all $\hm(f_s)$ are algebraic over $\ggv$. Hence,  
$\mu(f_I)>\delta$ would imply that $\hm(\phi)$ is algebraic over $\ggalg$, leading to $\hm(\phi)$ algebraic over $\ggv$ by Lemma \ref{alg}. This ends the proof that $\phi$ is $\mu$-minimal.\e

Let us now prove the last statement of the theorem.

For $a,b\in\kx_n$ with $\phi$-expansion given by (\ref{abc}), Proposition \ref{minimal0} shows that
$$
\mu(ab)=\Min\{\mu(c), \mu(d\phi)\},
$$
because $\phi$ is $\mu$-minimal. By equation (\ref{Hmu}), the inequality $\mu(c)\ge \mu(d\phi)$ implies
$$
\hm(ab)=
\begin{cases}
\hm(d)\hm(\phi),&\mbox{ if }\mu(c)>\mu(d\phi),\\
\hm(c)+\hm(d)\hm(\phi),&\mbox{ if }\mu(c)=\mu(d\phi).
\end{cases}
$$
By the minimality of $n$, the four elements $\hm(a)$, $\hm(b)$, $\hm(c)$, $\hm(d)$ are algebraic over $\ggv$. Hence, $\hm(\phi)$ would be algebraic over $\ggalg$, leading to $\hm(\phi)$ algebraic over $\ggv$ by Lemma \ref{alg}. This contradicts our assumption on $\hm(\phi)$.

Therefore, we must have $\mu(c)<\mu(d\phi)$, leading to $ab\smu c$.\e

Finally, let us prove that $\phi$ is $\mu$-irreducible.

Let $f,g\in\kx$ be polynomials such that $\phi\nmid_\mu f$, $\phi\nmid_\mu g$. By Proposition \ref{minimal0},
$$
\mu(f)=\mu(f_0),\qquad \mu(g)=\mu(g_0),
$$
where $f_0$, $g_0$ are the $0$-th degree coefficients of the $\phi$-expansions of $f$, $g$, respectively.

Let $f_0g_0=c+d\phi$ be the $\phi$-expansion of $f_0g_0$. As shown above, $f_0g_0\smu c$, so that
$$
\mu(fg)=\mu(f_0g_0)=\mu(c).
$$
Since $c$ is the $0$-th coefficient of the $\phi$-expansion of $fg$, the equality $\mu(fg)=\mu(c)$ shows that $\phi\nmid_\mu fg$, by Proposition \ref{minimal0}.
This ends the proof that $\phi$ is $\mu$-irreducible.
\end{proof}\e

\begin{corollary}\label{lmn}
Consider the following three natural numbers:\e

$\ell=$ minimal degree of $f\in\kx$ such that $\hm(f)$ is not a unit in $\ggm$.

$m=$ minimal degree of a key polynomial for $\mu$.

$n=$ minimal degree of $f\in\kx$ such that $\hm(f)$ is transcendental over $\ggv$.\e 

If one of these numbers exists, then all exist and \ $\ell=m=n$.
\end{corollary}

\begin{proof}
For any $f\in\kx$ we have
$$
f\in\kpm\imp \hm(f)\not\in\ggm^*\imp \hm(f)\not\in\ggalg\imp \kpm\ne\emptyset,  
$$
the last implication by Theorem \ref{phimindeg}. Hence, the conditions for the existence of these numbers are all equivalent:
\begin{align*}
\exists\hm(f)\not\in\ggm^*&\sii\kpm\ne\emptyset\sii \ggalg\subsetneq\ggm.
\end{align*}

Suppose these conditions are satisfied. By Theorem \ref{phimindeg}, there exists a key polynomial $\phi$ of degree $n$.
Also, for any $a\in\kx_n$, the homogeneous element $\hm(a)\in\ggm$ is algebraic over $\ggv$, hence a unit in $\ggm$ by (\ref{xi}). 

Since $\hm(\phi)$ is not a unit, this proves that $n=\ell$. Since there are no key polynomials in $\kx_n$, this proves $n=m$ too.
\end{proof}\e

\begin{corollary}\label{sprime+}
If $\phi$ is a key polynomial of minimal degree, then for all $f,g\in \kx$ we have
$$
\lrc(fg)=\lrc(f)\lrc(g),\qquad s'_\phi(fg)= s'_\phi(f)+s'_\phi(g).
$$ 
\end{corollary}

\begin{proof}
Consider the $\phi$-expansions $f=\sum_{0\le s}f_s\phi^s$, $g=\sum_{0\le t}g_t\phi^t$. 

We may write
$$
fg=\sum_{0\le j}b_j\phi^j,\qquad b_j=\sum_{s+t=j}f_sg_t.
$$
For each index $j$, there exist $s_0,t_0$ such that $s_0+t_0=j$ and
$$\mu(b_j\phi^j)\ge \mu(f_{s_0}\phi^{s_0}g_{t_0}\phi^{t_0})\ge \mu(fg),
$$
the last inequality because $\mu(f_{s_0}\phi^{s_0})\ge\mu(f)$ and $\mu(g_{t_0}\phi^{t_0})\ge \mu(g)$. Hence,  
$$
fg\smu \sum_{j\in J}b_j\phi^j,\qquad J=\left\{0\le j\mid \mu\left(b_j\phi^j\right)=\mu(fg)\right\}.
$$

For every $j\in J$, consider the set 
$$I_j=\{(s,t)\mid s+t=j,\ s\in I_\phi(f),\ t\in I_\phi(g)\}.$$ Then, (\ref{Hmu}) and Theorem \ref{phimindeg} show the existence of  $c_{s,t}\in\kx_n$ such that 
$$
\hm(b_j)=\sum_{(s,t)\in I_j}\hm(f_sg_t)=\sum_{(s,t)\in I_j}\hm(c_{s,t})=\hm(c_j),
$$
where $c_j=\sum_{(s,t)\in I_j}c_{s,t}$. 
Therefore,
again by  (\ref{Hmu}), we deduce that
$$
fg\smu h,\qquad h=\sum_{j\in J}c_j\phi^j.
$$
Note that $J=I_\phi(h)$ by construction. 

By Lemma \ref{sprime}, $I_\phi(fg)=I_\phi(h)=J$ and $\lrc(fg)=\lrc(h)$.

Thus, if $\ell=s'_\phi(f)$, $m=s'_\phi(g)$, we need to show that $\mx(J)=\ell+m$ and $\lrc(h)=\hm(f_\ell)\hm(g_m)$.

If $j>\ell+m$, then for all pairs $(s,t)$ with $s+t=j$ we have either $s>\ell$ or $t>m$. Thus, $\mu(f_sg_t\phi^j)=\mu(f_s\phi^sg_t\phi^t)>\mu(fg)$. So that $j\not\in J$.

For $j=\ell+m$, the same argument applies to all pairs $(s,t)$ with $s+t=j$, except for the pair $(s_\ell,t_m)$, for which $\mu(f_{s_\ell}g_{t_m}\phi^j)=\mu(fg)$. Therefore, $j=\ell+m$ is the maximal index in $J$ and $\lrc(h)=\hm(c_{\ell+m})=\hm(b_{\ell+m})=\hm( f_\ell g_m)$.    
\end{proof}

\subsection*{Units and maximal subfield of $\dm$}

\begin{proposition}\label{smallunits}
Let $\phi$ be a key polynomial of minimal degree $n$. 
For any non-zero $g\in\kx$, with $\phi$-expansion $g=\sum_{0\le s}g_s\phi^s$, the following conditions are equivalent.
\begin{enumerate}
\item $g\smu a$, for some $a\in\kx_n$.
\item $\hm(g)$ is algebraic over $\ggv$. 
\item $\hm(g)$ is a unit in $\ggm$. 
\item $s_\phi(g)=s'_\phi(g)=0$.
\item $g\smu g_0$.
\end{enumerate}
\end{proposition}

\begin{proof}
If $g\smu a$, for some $a\in\kx_n$, then $\hm(g)=\hm(a)$  is algebraic over $\ggv$ by Corollary \ref{lmn}.

If $\hm(g)$ is algebraic over $\ggv$ then $\hm(g)$ is a unit by (\ref{xi}).

If  $\hm(g)$ is a unit, there exists $f\in\kx$ such that $fg\smu 1$. By Lemma \ref{sprime}, $I_\phi(fg)=I_\phi(1)=\{0\}$, so that $s_\phi(fg)=s'_\phi(fg)=0$. By equation (\ref{multiplicative}) and Corollary \ref{sprime+}, we deduce that  $s_\phi(g)=s'_\phi(g)=0$.

This condition  $s_\phi(g)=s'_\phi(g)=0$ is equivalent to $I_\phi(g)=\{0\}$, which implies 
$g\smu g_0$ by Lemma \ref{sphi}. 

This proves (1) $\Longrightarrow$ (2) $\Longrightarrow$ (3)
$\Longrightarrow$ (4) $\Longrightarrow$ (5). 
Finally, (5) $\Longrightarrow$ (1) is obvious.
\end{proof}\bs


As a consequence of this characterization of homogeneous algebraic elements, the subfield $\kal\subset\dm$ of all elements in $\dm$ which are algebraic over $\ggv$ can be expressed as:
\begin{equation}\label{defkal}
\kal=\dm^*\cup\{0\}=\left\{\hm(a)\mid a\in\kx_n,\ \mu(a)=0\right\}\cup\{0\}.
\end{equation}
Since $\kal$ contains all units of $\dm$, it is the maximal subfield contained in $\dm$.

Since every $\xi\in\kal$ is homogeneous of degree zero, any monic homogeneous algebraic equation of $\xi$ over $\ggv$ has coefficients in the residue field $\kv$. Thus, $\kal$ coincides with the algebraic closure of $\kv$ in $\dm$.

\begin{proposition}\label{maxsubfield}
Let $\kal\subset\dm$ be the algebraic closure of $\kv$ in $\dm$.
For any key polynomial $\phi$ of minimal degree $n$, the mapping $\kal\hookrightarrow\dm\twoheadrightarrow k_\phi$ is an isomorphism. 
\end{proposition}

\begin{proof}
The restriction to $\kal$ of the onto mapping $\dm\to k_\phi$ described in Proposition \ref{maxideal} maps
$$\hm(a)\ \longmapsto\  a(\t)+\m_\phi,\qquad \forall a\in \kx_n.$$ Since the images cover all $k_\phi$, this mapping is an isomorphism between $\kal$ and $k_\phi$.
\end{proof}

\subsection*{Upper bound for weighted values}

Let us characterize $\mu$-minimality of a polynomial $f\in\kx$ in terms of its $\phi$-expansion.

\begin{proposition}\label{minimal}
Let $\phi$ be a key polynomial of minimal degree $n$. 
For any $f\in \kx$ with $\phi$-expansion \,$f=\sum_{s=0}^\ell\,f_s\phi^s$, $f_\ell\ne 0$,\, the following conditions are equivalent:
\begin{enumerate}
\item $f$ is $\mu$-minimal.
\item $\deg(f)=s_\phi'(f)n$. 
\item $\deg(f_\ell)=0$ and $\mu(f)=\mu\left(f_\ell\phi^\ell\right)$.
\end{enumerate}
\end{proposition}

\begin{proof}
Since $\deg(f) = \deg(f_\ell)+\ell n$ and $s'_\phi(f)\le\ell$, item (2) is equivalent to $\deg(f_\ell)=0$ and $s'_\phi(f)=\ell$. Thus, (2) and (3) are equivalent. \e

Let us deduce (3) from (1). Since $\deg(f-f_\ell\phi^\ell)<\deg(f)$, the $\mu$-minimality of $f$ implies that $f-f_\ell\phi^\ell$ cannot be $\mu$-equivalent to $f$. Hence, $\mu(f_\ell\phi^\ell)=\mu(f)$. 

In particular, $\ell=\mx\left(I_\phi(f)\right)$.
By Proposition \ref{smallunits}, $\hm(f_s)$ is a unit for all $f_s\ne0$. Take $b\in\kx$ and  $c_s\in \kx_n$ such that $b\,f_\ell\smu1$ and  $bf_s\sim_{\mu}c_s$, for all $0\le s<\ell$. 

If we denote $c_\ell=1$, Lemma \ref{sphi} and equation (\ref{Hmu}) show that
$$
bf\smu b\sum_{s\in I_\phi(f)}f_s\phi^s\imp\hm(bf)=\sum_{s\in I_\phi(f)}\hm(bf_s\phi^s)=\sum_{s\in I_\phi(f)}\hm(c_s\phi^s).
$$
Hence, $bf\smu g:=\sum_{s\in I_\phi(f)}c_s\phi^s$. 
Since $f\mmu g$ and $f$ is $\mu$-minimal, 
$$
\deg(f_\ell)+\ell n=\deg(f)\le\deg(g)=\ell n, 
$$
which implies $\deg(f_\ell)=0$.\e

Finally, let us deduce (1) from (2). Take non-zero $g,h\in \kx$ such that $g\sim_{\mu}fh$. By Lemma \ref{sprime} and Corollary \ref{sprime+}, 
$s'_\phi(g)=s'_\phi(fh)=s'_\phi(f)+s'_\phi(h)$, so that $$\deg(f)=s'_\phi(f)\deg(\phi)\le s'_\phi(g)\deg(\phi)\le \deg(g).$$
Thus, $f$ is $\mu$-minimal.
\end{proof}\e

\begin{corollary}\label{minimalpower}
Take $f\in \kx$ and $m$ a positive integer. Then, $f$ is $\mu$-minimal if and only if $f^m$ is $\mu$-minimal.
\end{corollary}

\begin{proof}
Let $\phi$ be a key polynomial of minimal degree. By Corollary \ref{sprime+}, $s'_\phi(f^m)=ms'_\phi(f)$. Hence, condition (2) of Proposition \ref{minimal} is equivalent for $f$ and $f^m$.
\end{proof}\bs

As another consequence of the criterion for $\mu$-minimality, we may introduce an important numerical invariant of a valuation on $\kx$ admitting key polynomials.

\begin{theorem}\label{bound}
Let $\phi$ be a key polynomial of minimal degree for a valuation $\mu$ on $\kx$.
Then, for any monic non-constant $f\in \kx$ we have
$$
\mu(f)/\deg(f)\le C(\mu):=\mu(\phi)/\deg(\phi),
$$
and equality holds if and only if $f$ is $\mu$-minimal.
\end{theorem}

\begin{proof}
Since $\phi$ and $f$ are monic, we may write
$$
f^{\deg(\phi)}=\phi^{\deg(f)}+h,\qquad \deg(h)<\deg(\phi)\deg(f).
$$
By Proposition \ref{minimal0}, $\mu\left(f^{\deg(\phi)}\right)\le \mu\left(\phi^{\deg(f)}\right)$, or equivalently, $\mu(f)/\deg(f)\le C(\mu)$.

By Proposition \ref{minimal}, equality holds if and only if $f^{\deg(\phi)}$ is $\mu$-minimal, and this is equivalent to $f$ being $\mu$-minimal, by Corollary \ref{minimalpower}. 
\end{proof}

\section{Structure of $\dm$ as a $k_v$-algebra}\label{secDelta}
In this section we determine the structure of $\dm$ as a $k_v$-algebra, and we derive some specific information about the extension $\km/\kv$.

In section \ref{subsecIncomm}, we deal with the case $\mu/v$ incommensurable. We show that $\dm=\km$. In this case, all key polynomials have the same degree and they are all $\mu$-equivalent.

In section \ref{subsecKPempty}, we assume that $\mu$ admits no key polynomials. We have again  $\dm=\km$. Also, we find several characterizations of the condition $\kpm=\emptyset$.  

In section \ref{subsecComm}, we deal with the case $\mu/v$ commensurable and $\kpm\ne\emptyset$, which corresponds to the classical situation in which $\mu$ is \emph{residually transcendental}.

In this case, $\dm$ is isomorphic to a polynomial ring in one indeterminate with coefficients in $\kal$, the algebraic closure of $\kv$ in $\km$. 

\subsection{Case $\mu/v$ incommensurable}\label{subsecIncomm}

We recall that $\mu/v$ incommensurable means $\qg_v\subsetneq \qgm$, or equivalently, $\op{rr}(\gm/\gv)>0$.

\begin{lemma}\label{s}
Suppose $\mu/v$ is incommensurable. Let $\phi\in\kx$ be a monic polynomial of minimal degree $n$ satisfying $\mu(\phi)\not\in\qg_v$. 
Then, for all $f,g\in\kx$ we have:
\begin{enumerate}
\item $f\smu a\phi^{s(f)}$, for some $a\in\kx_n$ and some unique integer $s(f)\ge0$.
\item $\phi$ is a key polynomial and $s(f)=s_\phi(f)=s'_\phi(f)$ for all $f\in\kx$.
\item $f$ is a $\mu$-unit if and only if $s(f)=0$.  
\item $f\mmu g$  if and only if $s(f)\le s(g)$.
\item $f$ is $\mu$-irreducible if and only if $s(f)=1$.
\item $f$ is $\mu$-minimal if and only if $\deg(f)=s(f)n$.
\end{enumerate}
\end{lemma}

\begin{proof}
Consider the $\phi$-expansion $f=\sum_{0\le s}f_s\phi^s$, with $f_s\in\kx_n$ for all $s$.

All monomials have diferent $\mu$-value. In fact, an equality
$$
\mu\left(f_s\phi^s\right)=\mu\left(f_t\phi^t\right) \imp \mu(f_s)-\mu(f_t)=(t-s)\mu(\phi), 
$$
is possible only for $s=t$ because $\mu(\phi)\not\in\qg_v$, and $\mu(f_s),\mu(f_t)$ belong to $\qg_v$ by the minimality of $n$. 

Hence, $f\smu a\phi^s$ for the monomial of least $\mu$-value. The unicity of $s$ follows from the same argument as above. This proves (1).\e

By Proposition \ref{minimal0}, $\phi$ is $\mu$-minimal. Let us show that $\phi$ is $\mu$-irreducible.

Consider $f,g\in\kx$ such that $\phi\nmu f$, $\phi\nmu g$. By item (1), $f\smu a$, $g\smu b$, for some $a,b\in \kx_n$; in particular, $$\mu(fg)=\mu(ab)=\mu(a)+\mu(b)\in\qg_v,$$ by the minimality of $n$. By item (1), $fg\smu c\phi^s$ for some $c\in\kx_n$ and an integer $s\ge0$. Since $\mu(\phi)\not\in\qg_v$, the condition $s\mu(\phi)=\mu(fg)-\mu(c)\in\qg_v$ leads to $s=0$, so that $fg\smu c$. Since $\phi$ is $\mu$-minimal, we have $\phi\nmu fg$.
Hence, $\phi$ is $\mu$-irreducible. 

Once we know that $\phi$ is a key polynomial, item (1) implies $I_\phi(f)=\{s(f)\}$ for all $f\in\kx$. Thus, $s(f)=s_\phi(f)=s'_\phi(f)$. This proves (2).\e

If $f$ is a $\mu$-unit, then $\phi\nmid_\mu f$, so that $s(f)=0$.

Conversely, if $s(f)=0$, then $\hm(f)=\hm(a)$ for some $a\in\kx_n$. Since 
the polynomials  $a$ and $\phi$ are coprime, they satisfy a B\'ezout identity
$$
ab+\phi\, d=1,\qquad \deg(b)<n,\quad \deg(d)<\deg(a)<n.
$$
Clearly, $ab=1-d\phi$ is the canonical $\phi$-expansion of $ab$. Since we cannot have $ab\smu d\phi$, because $\mu(\phi)\not\in\qg_v$, we must have $ab\smu 1$. This proves (3). \e

The rest of statements follow easily from (1), (2) and (3).
\end{proof}\bs



In particular, all results of the last section apply to our key polynomial $\phi$, because it is a key polynomial of minimal degree.

\begin{theorem}\label{mainincomm}
Suppose $\mu/v$ incommensurable. Let $\phi\in\kx$ be a monic polynomial of minimal degree $n$ satisfying $\mu(\phi)\not\in\qg_v$. Then,
\begin{enumerate}
\item $\phi$ is a key polynomial for $\mu$.
\item All key polynomials have degree $n$, and are $\mu$-equivalent to $\phi$. More precisely,
$$\kpm=\left\{\phi+a\mid a\in \kx_n,\ \mu(a)>\mu(\phi)\right\}.
$$
In particular, $\ggm$ has a unique homogeneous prime ideal $\hm(\phi)\ggm$. 
\item The natural inclusions determine equalities \
$\kal=\dm=\km$.
\item $\rr(\phi)=0$, and $\km$ is a finite extension of $\kv$, isomorphic to $k_\phi$.
\end{enumerate}
\end{theorem}

\begin{proof}
By items (2), (5) and (6) of Lemma \ref{s}, $\phi$ is a key polynomial for $\mu$, all key polynomials have degree $n$, and are all $\mu$-divisible by $\phi$. By Lemma \ref{mid=sim}, they are $\mu$-equivalent to $\phi$.
This proves (1) and (2).\e


Since $\op{rr}(\gm/\gv)>0$, the inequality in equation (\ref{ftalineq}) shows that $\km/\kv$ is an algebraic extension. Since $\kv\subset \dm\subset \km$, the ring $\dm$ must be a field. In particular, $\kal=\dm$ by the remarks preceding Proposition \ref{maxsubfield}.

Let us show that $\dm=\km$. An element in $\km^*$ is of the form
$$(g/h)+\m_\mu\in \km^*,$$ with $g,\,h\in\kx$ such that $\mu(g/h)=0$.

By item (1) of Lemma \ref{s}, from $\mu(g)=\mu(h)$ we deduce that
$$
g\smu a\phi^s,\qquad h\smu b\phi^s,
$$
for certain integer $s\ge0$ and polynomials $a,b\in\kx_n$ such that $\mu(a)=\mu(b)$. Thus,
$$
\dfrac gh+\m_\mu=\dfrac{a\phi^s}{b\phi^s}+\m_\mu=\dfrac{a}{b}+\m_\mu.
$$
By Lemma \ref{s}, $\hm(a)$, $\hm(b)$ are units in $\ggm$, so that $\hm(a)/\hm(b)\in\dm^*$ is mapped to $(g/h)+\m_\mu$ under the embedding $\dm\hookrightarrow \km$. This proves (3).\e

By Proposition \ref{maxideal}, $\rr(\phi)$ is the kernel of the onto mapping $\dm\twoheadrightarrow k_\phi$. Since $\dm$ is a field, $\rr(\phi)=0$ and this mapping is an isomorphism. 

This ends the proof of (4), because $k_\phi/\kv$ is a finite extension.
\end{proof}

\subsection{Valuations not admitting key polynomials}\label{subsecKPempty}

\begin{theorem}\label{deltaKPempty}
If $\kpm=\emptyset$, the canonical embedding $\dm\hookrightarrow \km$  is an isomorphism.
\end{theorem}

\begin{proof}
An element in $\km^*$ is of the form
$$(f/g)+\m_\mu\in \km^*,\qquad f,g\in\kx,\quad \mu(f/g)=0.$$

If $\kpm=\emptyset$, then Corollary \ref{lmn} shows that $\hm(f)$ and $\hm(g)$ are units in $\ggm$. Hence, $\hm(f)\hm(g)^{-1}$ is an element in $\dm$ whose image in $\km$ is $(f/g)+\m_\mu$.  
\end{proof}\e

\begin{theorem}\label{kpempty}
Let $\mu$ be a valuation on $\kx$ extending $v$. The following conditions are equivalent.
\begin{enumerate}
\item $\kpm=\emptyset$.
\item $\ggm$ is algebraic over $\ggv$.
\item Every non-zero homogeneous element in $\ggm$ is a unit.
\item $\mu/v$ is commensurable and $\km/\kv$ is algebraic.
\item The set of weighted values
$$W=\left\{\mu(f)/\deg(f)\mid f\in\kx\setminus K\ \mbox{monic}\right\}
$$
does not contain a maximal element.
\end{enumerate}
\end{theorem}

\begin{proof}
 By Corollary \ref{lmn}, conditions (1), (2) and (3) are equivalent.\e
 
Let us show that (1) implies (4). If $\kpm=\emptyset$, then $\mu/v$ is commensurable by Theorem \ref{mainincomm}, and $\km/\kv$ is algebraic by Theorems \ref{phimindeg} and \ref{deltaKPempty}.\e

Let us now deduce (5) from (4). Take $\phi$ an arbitrary monic polynomial in $\kx\setminus K$. Let us show that $\mu(\phi)/\deg(\phi)\in W$ is not an upper bound for this set.

Since $\qg_v=\qgm$, there exists a positive integer $e$ such that $e\mu(\phi)\in\qg_v$. Thus, there exists $a\in K^*$ such that $\mu(a\phi^e)=0$, so that $\hm(a\phi^e)\in\km^*$.

By hypothesis, this element is algebraic over $\kv$. Hence, $\hm(\phi^e)$ is algebraic over $\ggv$, and Lemma \ref{alg} shows that $\hm(\phi)$ is algebraic over $\ggv$. As mentioned in (\ref{xi}),  $\hm(\phi)$ is integral over $\ggv$. Consider an homogeneous equation
\begin{equation}\label{eqn}
\ep_0+\ep_1\hm(\phi)+\cdots+\ep_{m-1}\hm(\phi)^{m-1}+\hm(\phi)^m=0,
\end{equation}
with $\ep_0\,\dots,\ep_m$ homogeneous elements in $\ggv$ such that $\deg(\ep_i\hm(\phi)^i)=m\mu(\phi)$ for all indices $0\le i< m$ for which $\ep_i\ne0$.

By choosing $a_i\in K$ with $\hm(a_i)=\ep_i$ for all $i$, equation (\ref{eqn}) is equivalent to
$$
\mu\left(a_0+a_1\phi+\cdots+a_{m-1}\phi^{m-1}+\phi^m\right)>\mu(\phi^m)=m\mu(\phi).
$$
Hence, this monic polynomial $f=a_0+a_1\phi+\cdots+a_{m-1}\phi^{m-1}+\phi^m$ has a larger weighted value
$$
\mu(f)/\deg(f)>\mu(\phi^m)/\deg(f)=\mu(\phi)/\deg(\phi).
$$
Hence, the set $W$ contains no maximal element.\e

Finally, the implication (5)$\imp$(1) follows from Theorem \ref{bound}. 
\end{proof}


\subsection{Case $\mu/v$ commensurable and $\kpm\ne\emptyset$}\label{subsecComm}

\begin{theorem}\label{resfield}
Suppose $\mu/v$ commensurable and $\kpm\ne\emptyset$. The canonical embedding $\dm\hookrightarrow \km$  induces an isomorphism between the field of fractions of $\dm$ and $\km$.
\end{theorem}

\begin{proof}
Let $\chi\in\kx$ be an arbitrary key polynomial for $\mu$.

We must show that the induced morphism $\op{Frac}(\dm)\!\to\!\km$ is onto.

An element in $\km^*$ is of the form
$$(f/g)+\m_\mu\in \km^*,\qquad f,g\in\kx,\quad \mu(f/g)=0.$$

Set $\alpha=\mu(f)=\mu(g)\in\gm$. 
By Lemma \ref{subgroup}, $\gm=\gen{\gchi,\mu(\chi)}$; hence, we may write
$$
-\alpha=\beta+s\mu(\chi),\qquad \beta\in\gchi,\qquad s\in\Z.
$$
Since $\mu(\chi)\in\Q\gchi$, we may assume that $0\le s<e$ for some positive integer $e$. Take  $a\in \kxchi$ such that $\mu(a)=\beta$. Then, the polynomial $h=a\phi^s$ satisfies $\mu(h)=-\alpha$.

Thus, $\hm(hf),\hm(hg)$ belong to $\dm$ and the fraction $\hm(hf)/\hm(hg)$ is mapped to $(f/g)+\m_\mu$ by the morphism $\op{Frac}(\dm)\!\to\!\km$.
\end{proof}\bs

\begin{theorem}\label{Dstructure}
Suppose $\mu/v$ commensurable. Let $\phi$ be a key polynomial of minimal degree $n$, and let $e$ be a minimal positive integer such that $e\mu(\phi)\in\gphi$. 

Take $u\in\kx_n$ such that  $\mu(u\phi^e)=0$. Then, $\xi=\hm(u\phi^e)\in\dm$ is transcendental over $\kv$ and $\dm=\kal[\xi]$.  
\end{theorem}

\begin{proof}
The element $\xi$ is not a unit, because it is divisible by the prime element $\hm(\phi)$. By (\ref{xi}), $\xi$ is transcendental over $\kv$.\e

Consider $\hm(g)\in\dm$, for some  $g\in\kx$ with $\mu(g)=0$. Let $g=\sum_{0\le s}g_s\phi^s$  be the $\phi$-expansion of $g$. 

Let $I=I_\phi(g)$ be the set of indices $s$ such that  $\mu(a_s\phi^s)=0$. For each $s\in I$, the equality $s\mu(\phi)=-\mu(a_s)\in\gphi$ implies that $s=ej_s$ for some integer $j_s\ge0$.

By Proposition \ref{minimal0} and equation (\ref{Hmu}), 
\begin{equation}\label{usual}
g\smu \sum\nolimits_{s\in I}g_s\phi^s,\qquad\hm(g)=\sum\nolimits_{s\in I}\hm(g_s\phi^s).
\end{equation}

For each $s\in I$, Proposition \ref{smallunits} shows that $\hm(u)$ is a unit, and there exists $c_s\in\kx_n$ such that $\hm(c_s)=\hm(g_s)\hm(u)^{-j_s}$. Hence,
$$
\hm(g_s\phi^s)=\hm(g_s)\hm(u)^{-j_s}\hm(u)^{j_s}\hm(\phi^s)=\hm(c_s)\xi^{j_s}\in \kal[\xi].
$$
Hence, $\hm(g)\in \kal[\xi]$. This proves that $\dm=\kal[\xi]$.  
\end{proof}\bs

As a consequence of Theorems \ref{mainincomm}, \ref{deltaKPempty}, \ref{kpempty}, \ref{resfield} and \ref{Dstructure}, we obtain the following computation of the residue class field $\km$.

\begin{corollary}\label{kstructure}
If $\kpm=\emptyset$, then $\kal=\dm=\km$ is an algebraic extension of $k$.

If $\mu/v$ is incommensurable, then, $\kal=\dm=\km$ is a finite extension of $k$.

If $\mu/v$ is commensurable and  $\kpm\ne\emptyset$, then $\km\simeq\kal(y)$, where $y$ is an indeterminate.
\end{corollary}

\section{Residual polynomial operator}\label{secR} 
Suppose  $\mu/v$ commensurable and $\kpm\ne\emptyset$. 

Let us fix a key polynomial $\phi\in\kpm$ of minimal degree $n$. Let $q=\hm(\phi)$ be the corresponding prime element of $\ggm$. \e

For the description of the set $\kpm$ in section \ref{secKPuf}, we need a ``residual polynomial" operator $R\colon \kx\to\kal[y]$, yielding a decomposition of any homogeneous element $\hm(f)\in\ggm$ into a product of a unit, a power of $q$, and the degree-zero element $R(f)(\xi)\in\Delta=\kal[\xi]$ (Theorem \ref{Hmug}). As a consequence, the operator $R$ provides a computation of the residual ideal operator (Theorem \ref{RR}).\e

In Lemma \ref{subgroup} we proved that $\gm=\gen{\g_n,\mu(\phi)}$, where $\g_n$ is the subgroup 
$$
\g_n=\{\mu(a)\mid a\in\kx_n,\ a\ne0\}\subset\gm.
$$
Let $e$ be a minimal positive integer with $e\mu(\phi)\in\g_n$. By Theorem \ref{bound}, all key polynomials $\chi$ of degree $n$ have the same $\mu$-value $\mu(\chi)=\mu(\phi)$. Thus, this positive integer $e$ does not depend on the choice of $\phi$. 

It will be called the \emph{relative ramification index of $\mu$}.\e

We fix a polynomial $u\in\kx_n$ such that $\mu(u\phi^e)=0$, and consider $$\xi=\hm(u\phi^e)=\hm(u)q^e\in\dm.$$ 

By Theorem \ref{Dstructure}, $\xi$ is transcendental over $\kv$ and $\dm=\kal[\xi]$.\e

Throughout this section, for any polynomial $f\in\kx$ we denote
$$
s(f):=s_\phi(f),\qquad s'(f):=s'_\phi(f),\qquad I(f):=I_\phi(f).
$$

For $s\in I(f)$, the condition $\mu(f_s\phi^s)=\mu(f)$ implies that $s$ belongs to a fixed class modulo $e$. In fact,
for any pair $s,t\in I(f)$,
$$
\mu(f_s\phi^s)=\mu(f_t\phi^t)\imp (t-s)\mu(\phi)=\mu(f_s)-\mu(f_t)\in\g_n \imp t\equiv s\md{e}. 
$$
Hence,  $I(f)\subset\left\{s_0,s_1,\dots,s_d\right\}$, where
$$
s_0=s(f)=\Min(I(f)),\quad s_j=s_0+je,\ 0\le j\le d,\quad s_d=s'(f)=\mx(I(f)). 
$$

By Lemma \ref{sphi}, we may write
\begin{equation}\label{rdetre}
f\smu\sum_{s\in I(f)}f_s\phi^s\smu \phi^{s_0}\left(f_{s_0}+\cdots +f_{s_j}\phi^{je}+\cdots +f_{s_d}\phi^{de}\right),
\end{equation}
having into account only the monomials for which $s_j\in I(f)$.

\begin{definition}
Consider the residual polynomial operator
$$
R:=R_\phi\colon \kx \lra \kal[y],\qquad R(f)=\zeta_0+\zeta_1y+\cdots+\zeta_{d-1}y^{d-1}+y^d,
$$
for $f \ne0$, where the coefficients $\zeta_j\in\kal$ are defined by:
\begin{equation}\label{defzetaj}
\zeta_j=\begin{cases}
\hm(f_{s_d})^{-1}\hm(u)^{d-j}\hm(f_{s_j}),&\mbox{ if }s_j\in I(f),\\
0,&\mbox{ if }s_j\not\in I(f).
\end{cases}
\end{equation}

Also, we define $R(0)=0$. 
\end{definition}

For $s_j\in I(f)$, we have 
$$\mu(f_{s_j}\phi^{je})=\mu(f_{s_d}\phi^{de})=\mu(f/\phi^{s_0}),$$
so that $\mu(f_{s_j})=\mu(f_{s_d})+(d-j)e\mu(\phi)=\mu(f_{s_d})-(d-j)\mu(u)$. 

Since the three homogeneous elements $\hm(f_{s_j})$, $\hm(f_{s_d})$ and $\hm(u)$ are units in $\ggm$, we deduce that $\zeta_j\in\dm^*=(\kal)^*$ for $s_j\in I(f)$. 

Thus, the monic residual polynomial $R(f)$ is well defined, and it has degree
\begin{equation}\label{degR}
d(f):=\deg(R(f))=d=\left(s'(f)-s(f)\right)/e.
\end{equation}

Note that $\zeta_0\ne0$, because $s_0\in I(f)$. Thus, $R(f)(0)\ne0$.\bs

\noindent{\bf Example. }For any monomial $f=a\phi^s$ with $a\in\kx_n$, we have $R(f)=1$.

\begin{definition}\label{defnlc}
With the above notation, the \emph{normalized leading residual coeficient}
$$
\nlc(f)=\hm(f_{s_d})\hm(u)^{-d}=\lrc(f)\hm(u)^{-d}\in\ggm^*,
$$
is an homogeneous element in $\ggm$ of degree $\mu(f)-s(f)\mu(\phi)$.
\end{definition}

For any $g\in\kx$, from (\ref{multiplicative}) and Corollary \ref{sprime+}, we deduce that 
$$
d(fg)=d(f)+d(g),\qquad \nlc(fg)=\nlc(f)\nlc(g),\quad \forall f,g\in\kx.
$$

By definition, for any $s_j\in I(f)$ we have
$$\nlc(f)\,\zeta_j\,\xi^j=\hm(f_{s_j})\hm(\phi^{je}).
$$
Thus,  (\ref{rdetre}) leads to the following identity, which is the ``raison d'$\hat{e}$tre" of $R(f)$.

\begin{theorem}\label{Hmug}
For any $f\in\kx$, we have $\hm(f)=\nlc(f)\,q^{s(f)}R(f)(\xi)$.\hfill{$\Box$}
\end{theorem}

Note that $\nlc(f)$ is a unit, $q^{s(f)}$ the power of a prime element, and $R(f)(\xi)\in\dm$.

Let us derive from Theorem \ref{Hmug} some basic properties of the residual polynomials.

\begin{corollary}\label{Rmult}
For all $f,g\in\kx$, we have $R(fg)=R(f)R(g)$.
\end{corollary}

\begin{proof}
Since the functions $\hm$ and $\nlc$ are multiplicative, Theorem \ref{Hmug} shows that $$R(fg)(\xi)=R(f)(\xi)R(g)(\xi).$$
By Theorem \ref{Dstructure}, we deduce that $R(fg)=R(f)R(g)$. 
\end{proof}\e

\begin{corollary}\label{equivR}
For all $f,g\in K[x]$,  
$$
\as{1.3}
\begin{array}{ccl}
f \smu g&\iff&I(f)=I(g), \ \nlc(f)=\nlc(g)\ \mbox{ and }\ R(f)=R(g).\\
f \mmu g&\iff&s(f)\le s(g) \ \mbox{ and }\ R(f)\mid R(g) \ \mbox{ in }\ \kal[y].
\end{array}
$$
\end{corollary}

\begin{proof}
If $f \smu g$, then $I(f)=I(g)$ and $\nlc(f)=\nlc(g)$ by Lemma \ref{sprime}.  Thus, $R(f)(\xi)=R(g)(\xi)$ by Theorem \ref{Hmug}, leading to
$R(f)=R(g)$ by Theorem \ref{Dstructure}.

Conversely, $I(f)=I(g)$ implies $s(f)=\Min(I(f))=\Min(I(g))=s(g)$. Thus, $\hm(f)=\hm(g)$ follows from Theorem \ref{Hmug}. \e

If $f \mmu g$, then $fh\smu g$ for some $h\in\kx$. By the first item and Corollary \ref{Rmult}, we get $R(g)=R(fh)=R(f)R(h)$, so that $R(f)\mid R(g)$. 

Also, since $s(g)=s(f)+s(h)$, we deduce that $s(f)\le s(g)$.

Conversely,  $s(f)\le s(g)$ and $R(f)\mid R(g)$ imply $\hm(f)\mid\hm(g)$ by Theorem \ref{Hmug}, having in mind that $\nlc(f)$, $\nlc(g)$ are units in $\ggm$.
\end{proof}\e

\begin{corollary}\label{Rconstruct}
Let $s\in\Z_{\ge0}$, $\zeta\in (\kal)^*$, and $\psi\in\kal[y]$ a monic polynomial with $\psi(0)\ne0$. 
Then, there exists a polynomial $f\in\kx$ such that
$$
s(f)=s,\qquad \nlc(f)=\zeta,\qquad R(f)=\psi.
$$
\end{corollary}

\begin{proof}
Let $\psi=\zeta_0+\zeta_1y+\cdots \zeta_{d-1}y^{d-1}+\zeta_dy^d$, with $\zeta_0,\dots,\zeta_{d-1}\in\kal$ and $\zeta_d=1$. Let $I$ be the set of indices $0\le j\le d$ with $\zeta_j\ne0$. 

By (\ref{defkal}), for each $j\in I$ we may take $f_j\in \kx_n$ such that $\hm(f_j)=\zeta\hm(u)^j\zeta_j$. Then, 
$f=\phi^s\left(f_0+\cdots+f_j\phi^{je}+\cdots +f_d\phi^{de}\right)$
satisfies all our requirements.
\end{proof}\e

\begin{theorem}\label{RR}
For any non-zero $f\in K[x]$, 
$$\rr(f)=\xi^{\lceil s(f)/e\rceil}R(f)(\xi)\Delta.
$$
\end{theorem}

\begin{proof}
By definition, an element in the ideal $\rr(f)$ is of the form $\hm(h)$ for some $h\in\kx$ such that $f\mmu h$ and $\mu(h)=0$. 

The condition $\mu(h)=0$ implies $e\mid s(h)$. By Theorem \ref{Hmug}, 
$$\hm(h)=\xi^{s(h)/e}\hm(u)^{-s(h)/e}\nlc(h)\,R(h)(\xi).$$
On the other hand, Corollary \ref{equivR} shows that $s(f)\le s(h)$ and $R(f)\mid R(h)$. Therefore, $\hm(h)$ belongs to the ideal $\xi^{\lceil s(f)/e\rceil}R(f)(\xi)\Delta$.

Conversely, if $m=\lceil s(f)/e\rceil$, then Theorem \ref{Hmug} shows that 
$$\xi^mR(f)(\xi)=q^{me}\hm(u)^mR(f)(\xi)=\hm(f)q^{me-s(f)}\nlc(f)^{-1}\hm(u)^m\in\rr(f),$$
because $me\ge s(f)$ and $\nlc(f)$, $\hm(u)$ are units. 
\end{proof}

\subsection{Dependence of $R$ on the choice of $u$}
Let $u^*\in\kx_n$ be another choice of a polynomial such that $\mu(u^*\phi^e)=0$, and denote
$$
\xi^*=\hm(u^*\phi^e)=\hm(u^*)q^e\in\dm.
$$ 

Since $\mu(u)=\mu(u^*)$, and $\hm(u)$, $\hm(u^*)$ are units in $\ggm$, we have $$\xi^*=\sigma^{-1}\xi,\quad \mbox{ where }\sigma=\hm(u)\hm(u^*)^{-1}\in\dm^*=(\kal)^*.$$

Let $R^*$ be the residual polynomial operator associated with this choice of $u^*$.

For any $f\in\kx$, suppose that $R^*(f)=\zeta^*_0+\zeta^*_1y+\cdots + \zeta^*_{d-1}y^{d-1}+y^d$. By the very definition (\ref{defzetaj}) of the residual coefficients, 
$$
\zeta^*_j=\sigma^{d-j}\zeta_j,\qquad 1\le j\le d.
$$

We deduce the following relationship between $R$ and $R^*$:
$$
R^*(f)(y)=\sigma^dR(f)(\sigma^{-1} y),\qquad \forall f\in\kx.
$$

\subsection{Dependence of $R$ on the choice of $\phi$}

Let $\phi_*$ be another key polynomial with minimal degree $n$, and denote 
$q_*=\hm(\phi_*)$. 

By Theorem \ref{bound}, $\mu(\phi_*)=\mu(\phi)$, so that
$$
\phi_*=\phi+a,\qquad a\in\kx_n,\quad \mu(a)\ge\mu(\phi).
$$

In particular, $\mu(u\phi_*^e)=0$, and we may consider 
$$
\xi_*=\hm(u\phi_*^e)=\hm(u)q_*^e\in\dm
$$
as a transcendental generator of $\dm$ as a $\kal$-algebra. 

Let $R_*$ be the residual polynomial operator associated with this choice of $\phi_*$.

\begin{proposition}\label{lastlevel}
Let $\phi_*$ be another key polynomial with minimal degree, and denote with a subindex $(\ )_*$ all objects depending on $\phi_*$.
\begin{enumerate}
\item If $\phi_*\smu\phi$, then \ $q_*=q$, \ $\xi_*=\xi$  \,and\, $R_*=R$.
\item If $\phi_*\not\smu\phi$, then \ $e=1$, \ $q_*=q+\hm(a)$  \,and\,\, $\xi_*=\xi+\tau$, 

\noindent where \,$\tau=\hm(ua)\in(\kal)^*$. In this case, for any $f\in K[x]$ we have
\begin{equation}\label{RR*}
y^{s(f)}R(f)(y)=(y+\tau)^{s_*(f)}R_*(f)(y+\tau).
\end{equation}
In particular, $s_*(f)=\ord_{y+\tau}\left(R(f)\right)$ and $s(f)+d(f)=s_*(f)+d_*(f)$.
\end{enumerate}
\end{proposition}

\begin{proof}
Suppose $\phi_*\smu\phi$. By definition,
$$
q_*=\hm(\phi_*)=\hm(\phi)=q,\qquad \xi_*=\hm(u\phi_*)=\hm(u\phi)=\xi.
$$

Let $f\in\kx$. By , we can replace $\phi$ with $\phi_*$ in equation (\ref{rdetre}) to obtain
$$
f\smu\sum\nolimits_{s\in I(f)}f_s\phi_*^s\smu \phi_*^{s_0}\left(f_{s_0}+\cdots +f_{s_j}\phi_*^{je}+\cdots +f_{s_d}\phi_*^{de}\right).
$$
Hence, (\ref{defzetaj}) leads to the same residual coefficients, so that $R(f)=R_*(f)$.\e

Suppose $\phi_*\not\smu\phi$; that is, $\mu(a)=\mu(\phi)$. Then,  $e=1$, $\hm(\phi_*)=\hm(\phi)+\hm(a)$, and $$\xi_*=\hm(u\phi_*)=\hm(u)\hm(\phi)+\hm(u)\hm(a)=\xi+\tau.
$$

Finally, let $f\in\kx$, and denote $s=s(f)$, $s_*=s_*(f)$. By Theorem \ref{Hmug},
$$
q^s\,\nlc(f)R(f)(\xi)=\hm(f)=q_*^{s_*}\,\nlc_*(f)R_*(f)(\xi_*).
$$
Since $q=\hm(u)^{-1}\xi$ and $q_*=\hm(u)^{-1}\xi_*=\hm(u)^{-1}(\xi+\tau)$, we deduce 
$$
\xi^sR(f)(\xi)=\sigma(\xi+\tau)^{s_*}R^*(f)(\xi+\tau),
$$
where $\sigma=\hm(u)^{s-s_*}\nlc_*(f)\nlc(f)^{-1}\in(\kal)^*$, because 
$$\deg\left(\nlc(f)\hm(u)^{-s}\right)=0=\deg\left(\nlc_*(f)\hm(u)^{-s_*}\right).$$

By Theorem \ref{Dstructure}, this implies $y^sR(f)(y)=\sigma(y+\tau)^{s_*}R^*(f)(y+\tau)$. Since $R(f)$ and $R_*(f)$ are monic polynomials, we have necessarily $\sigma=1$. This proves (\ref{RR*}).
\end{proof}

\section{Key polynomials and unique factorization in $\ggm$}\label{secKPuf}
We keep assuming $\mu/v$ commensurable and $\kpm\ne\emptyset$. Also, we keep dealing with a fixed key polynomial $\phi$ of minimal degree $n$, and we denote
$$
s(f):=s_\phi(f),\qquad s'(f):=s'_\phi(f),\qquad R(f):=R_\phi(f), \qquad \forall f\in\kx.
$$

\subsection{Homogeneous prime elements}\label{subsecHomogP}

By Theorem \ref{Dstructure}, the prime elements in $\dm$ are those of the form $\psi(\xi)$ for  $\psi\in\kal[y]$ an irreducible polynomial. 

An element in $\dm$ which is a prime in $\ggm$, is a prime in $\dm$, but the converse is not true. Let us now discuss what primes in $\dm$ remain prime in $\ggm$.

\begin{lemma}\label{prime0}
Let $\psi\in \kal[y]$ be a monic irreducible polynomial.
\begin{enumerate}
\item   If $\psi\ne y$, then  $\psi(\xi)$ is a prime element in $\ggm$. 
\item If $\psi=y$, then $\xi$ is a prime element in $\ggm$ if and only if $e=1$. 
\end{enumerate}
\end{lemma}

\begin{proof}
Suppose $\psi\ne y$. Being $\psi$ irreducible, we have $\psi(0)\ne0$. By Lemma \ref{Rconstruct}, there exists $f\in\kx$ such that $s(f)=0$, $\nlc(f)=1$ and $R(f)=\psi$. By Theorem \ref{Hmug}, $\hm(f)=\psi(\xi)$.

Suppose $\psi(\xi)=\hm(f)$ divides the product of two homogeneous elements in $\ggm$. Say $f\mmu gh$ for some $g,h\in\kx$.

By Corollaries \ref{equivR} and \ref{Rmult}, $\psi=R(f)$ divides $R(gh)=R(g)R(h)$. Being $\psi$ irreducible, it divides either $R(g)$ or $R(h)$, and this leads to $\psi(\xi)$ dividing either $\hm(g)$ or $\hm(h)$ in $\ggm$, by Theorem \ref{Hmug}.\e

The element $\xi$ is associate to $q^{e}$ in $\ggm$. Since $q$ is a prime element, its $e$-th power is a prime if and only if $e=1$.   
\end{proof}\bs

Besides these prime elements belonging to $\dm$, we know that  $q$ is another prime element in $\ggm$, of degree $\mu(\phi)$. 

The next result shows that there are no other homogeneous prime elements in $\ggm$, up to multiplication by units.  

\begin{proposition}\label{mu-irr}
A polynomial $f\in\kx$ is $\mu$-irreducible if and only if one of the two following conditions is satisfied:
\begin{enumerate}
\item[(a)] $s(f)=s'(f)=1$.
\item[(b)] $s(f)=0$ and $R(f)$ is irreducible in $\kal[y]$.
\end{enumerate}
In the first case, $\hm(f)$ is associate to $q$. In the second case, to $R(f)(\xi)$.
\end{proposition}

\begin{proof}
By Theorem \ref{Hmug}, $\hm(f)=q^{s(f)}\,\nlc(f)\,R(f)(\xi)$. Since $\nlc(f)$ is a unit and $q$ is a prime, $\hm(f)$ is a prime if and only if one of the two following conditions is satisfied:
\begin{enumerate}
\item[(i)] $s(f)=1$ and $R(f)(\xi)$ is a unit.
\item[(ii)] $s(f)=0$ and $R(f)(\xi)$ is a prime in $\ggm$.
\end{enumerate}

The homogeneous element of degree zero $R(f)(\xi)$ is a unit in $\ggm$ if and only if it is a unit in $\dm$. 
By Theorem \ref{Dstructure}, this is equivalent to $\deg(R(f))=0$, which in turn is equivalent to $s(f)=s'(f)$ by (\ref{degR}). Thus, (i) is equivalent to (a), and $\hm(f)$ is associate to $q$ in this case.

Since $R(f)\ne y$, (ii) is equivalent to (b) by Lemma \ref{prime0}. Clearly, $\hm(f)$ is associate to  $R(f)(\xi)$ in this case.
\end{proof}\bs

Putting together this characterization of $\mu$-irreducibility with the characterization of $\mu$-minimality from Proposition \ref{minimal}, we get the following characterization of key polynomials.

\begin{proposition}\label{charKP}
Let $\phi$ be a key polynomial for $\mu$, of minimal degree $n$.

A monic $\chi\in\kx$ is a key polynomial for $\mu$ if and  only if one of the two following conditions is satisfied:
\begin{enumerate}
\item[(a)] $\deg(\chi)=\deg(\phi)$ \,and\; $\chi\smu\phi$.
\item[(b)] $s(\chi)=0$, $\deg(\chi)=e\,n\deg(R(\chi))$\, and\; $R(\chi)$ is irreducible in $\kal[y]$.
\end{enumerate}

In the first case, $\rr(\chi)=\xi\dm$.
In the second case, $\rr(\chi)=R(\chi)(\xi)\dm$.
\end{proposition}

\begin{proof}
If $\chi$ satisfies (a), then $\chi$ is a key polynomial by Lemma \ref{mid=sim}. 

Also, $\rr(\chi)=\rr(\phi)=\xi\dm$ by Theorem \ref{RR}, since $s(\phi)=1$ and $R(\phi)=1$.\e 

If $\chi$ satisfies (b), then $\deg(R(\chi))=s'(\chi)/e$ by (\ref{degR}), so that $\deg(\chi)=s'(\chi)n$, and $\chi$ is $\mu$-minimal by Proposition \ref{minimal}.
Also, $\chi$  is $\mu$-irreducible by Proposition \ref{mu-irr}.

Thus, $\chi$ is a key polynomial, and $\rr(\chi)=R(\chi)(\xi)\dm$ by Theorem \ref{RR}. \e

Conversely, suppose $\chi$ is a key polynomial for $\mu$. Since $\chi$ is $\mu$-minimal, it has 
$$
\deg(\chi)=s'(\chi)n,\qquad \mu(\chi)=\mu(\phi),
$$ 
by Proposition \ref{minimal}. 

Since $\chi$ is $\mu$-irreducible, it satisfies one of the conditions of Proposition \ref{mu-irr}.   \e

If $s(\chi)=s'(\chi)=1$, we get $\deg(\chi)=n$ and $\phi\mmu \chi$. Thus, $\chi\smu\phi$ by Lemma \ref{mid=sim}, and $\phi$ satisfies (a).\e

If $s(\chi)=0$ and $R(\chi)$ is irreducible in $\kal[y]$, then  $\deg(R(\chi))=s'(\chi)/e$ by (\ref{degR}). Thus, $\deg(\chi)=s'(\chi)n=en\deg(R(\chi))$, and $\chi$ satisfies (b).
\end{proof}

\begin{corollary}\label{gchi}
Let $\phi$ be a key polynomial for $\mu$, of minimal degree $n$.

Let $\chi\in\kx$ be a key polynomial such that $\chi\not\smu\phi$. Then, $\gchi=\gm$.
\end{corollary}

\begin{proof}
By Lemma \ref{subgroup}, $\gm=\gen{\gchi,\mu(\chi)}$. Since $\deg(\chi)\ge n$, we clearly have $\g_n\subset \gchi$. 
By Theorem \ref{bound} and Proposition \ref{charKP},
$$
\mu(\chi)=\dfrac{\deg(\chi)}{\deg(\phi)}\,\mu(\phi)=\deg(R(\chi))e\mu(\phi)\in\g_n\subset \gchi.
$$
Hence, $\gm=\gchi$.  
\end{proof}

\begin{corollary}\label{e>1}
The two following conditions are equivalent.
\begin{enumerate}
\item $e>1$
\item All key polynomials of minimal degree are $\mu$-equivalent.
\end{enumerate}
\end{corollary}

\begin{proof}
Let $\phi$ be a key polynomial of minimal degree $n$.\e

If $e>1$ and $\chi$ is a key polynomial not $\mu$-equivalent to $\phi$, then Proposition \ref{charKP} shows that $\deg(\chi)=e\,n\deg(R(\chi))>n$. Hence, all key polynomials of degree $n$ are $\mu$-equivalent to $\phi$.\e 

If $e=1$, then $\mu(\phi)\in \g_n$, so that there exists $a\in \kx_n$ with $\mu(a)=\mu(\phi)$. The monic polynomial $\chi=\phi+a$ has degree $n$ and it is not $\mu$-equivalent to $\phi$. Also, $\deg(R(\chi))=1$, so that $R(\chi)$ is irreducible. Therefore, $\chi$ is a key polynomial for $\mu$, because it satisfies condition (b) of Proposition \ref{charKP}.  
\end{proof}

\subsection{Unique factorization in $\ggm$}\label{subsecKPUF}

If $\chi$ is a key polynomial for $\mu$, then $\rr(\chi)$ is a maximal ideal of $\dm$, by Proposition \ref{maxideal}.   
Let us study the fibers of the mapping $\rr\colon \kpm\to\mx(\dm)$.   

\begin{proposition}\label{samefiber}
Let $\phi$ be a key polynomial of minimal degree, and let $R=R_\phi$.

For any $\chi,\chi'\in\kpm$, the following conditions are equivalent:
\begin{enumerate}
\item $\chi\smu\chi'$.
\item $\hm(\chi)$ and $\hm(\chi')$ are associate in $\ggm$. 
\item $\chi\mmu\chi'$. 
\item $\rr(\chi)=\rr(\chi')$.
\item $R(\chi)=R(\chi')$.
\end{enumerate}

Moreover, these conditions imply \,$\deg(\chi)=\deg(\chi')$.
\end{proposition}

\begin{proof}
The implications (1) $\imp$ (2) $\imp$ (3) are obvious.

Also, (3) $\imp$ $\rr(\chi')\subset \rr(\chi)$ $\imp$  (4), because $\rr(\chi')$ is a maximal ideal.    \e

Let us show that (4) implies (5). By Proposition \ref{charKP}, condition (4) implies that we have two possibilities for the pair $\chi$, $\chi'$:\e

(i) \ $\chi\smu\phi\smu\chi'$, or \e

(ii) \ $s(\chi)=s(\chi')=0$.\e

In the first case, we deduce (5) from Corollary \ref{equivR}. 

In the second case, condition (5) follows from Theorems \ref{RR}, \ref{Dstructure}, and the fact that $R(\chi)$, $R(\chi')$ are monic polynomials.\e

Let us show that (5) implies (1).  If $R(\chi)=R(\chi')=1$, then Proposition \ref{charKP} shows that $\chi\smu\phi\smu\chi'$.

If $R(\chi)=R(\chi')\ne1$, then Proposition \ref{charKP} shows that $s(\chi)=s(\chi')=0$ and
$$
\deg(\chi)=en\deg(R(\chi))=en\deg(R(\chi'))=\deg(\chi').
$$
Also, $\chi\mmu\chi'$ by item (2) of Corollary \ref{equivR}. Hence, 
$\chi\smu\chi'$ by Lemma \ref{mid=sim}. 

This ends the proof of the equivalence of all conditions.\e

Finally, (1) implies $\deg(\chi)=\deg(\chi')$ by the $\mu$-minimality of both polynomials. 
\end{proof}\e

\begin{theorem}\label{Max}
Suppose $\kpm\ne\emptyset$. The residual ideal mapping 
$$\rr\colon \kpm\,\lra\,\mx(\dm)$$ induces a bijection  
between $\kpm/\!\!\smu$ and\, $\mx(\dm)$.
\end{theorem}

\begin{proof}
If $\mu/v$ is incommensurable, the statement follows from Theorem \ref{mainincomm}.\e

Suppose $\mu/v$ commensurable, and let $\phi$ be a key polynomial of minimal degree $n$.

By Proposition \ref{samefiber},  $\rr$ induces a 1-1 mapping between $\kpm/\!\!\smu$  and $\mx(\dm)$. 

Let us show that $\rr$ is onto. By Theorem \ref{Dstructure}, a maximal ideal in $\Delta$ is given by $\psi(\xi)\dm$ for some monic irreducible polynomial $\psi\in\kal[y]$.

If $\psi=y$, then $\ll=\rr(\phi)$, by Theorem \ref{RR}. 
If $\psi\ne y$, then it suffices to show the existence  of a key polynomial $\chi$ such that $R(\chi)=\psi$, by Proposition \ref{charKP}.  

Let $d=\deg(\psi)$. By Lemma \ref{Rconstruct}, there exists $\chi\in\kx$ such that $s(\chi)=0$, $\nlc(\chi)=\hm(u)^{-d}$ and $R(\chi)=\psi$.
Along the proof of that lemma, we saw that $\chi$ may be chosen to have $\phi$-expansion:
$$
\chi=a_0+a_1\phi^e+\cdots+a_d\phi^{de}, \qquad \deg(a_j)<n.
$$
Also, the condition on $a_d$ is $\hm(a_d)=\nlc(\chi)\hm(u)^d=1_{\ggm}$. Thus, we may choose $a_d=1$. Then $\deg(\chi)=den$, so that $\chi$ is a key polynomial because it satisfies condition (b) of Proposition \ref{charKP}. 
\end{proof}\e

\begin{theorem}\label{homogeneousprimes}
Let $\pset\subset\kpm$ be a set of representatives of key polynomials under $\mu$-equivalence. 
Then, the set $H\pset=\left\{\hm(\chi)\mid \chi\in \pset\right\}$ is a system of representatives of homogeneous prime elements of $\ggm$ up to associates.  

Also, up to units in $\ggm$, for any non-zero $f\in K[x]$, there is a unique factorization:
\begin{equation}\label{factorization}
f\smu \prod\nolimits_{\chi\in \pset}\chi^{a_\chi},\quad  a_\chi=s_\chi(f).
\end{equation}
\end{theorem}

\begin{proof}
All elements in $H\pset$ are prime elements by the definition of $\mu$-irreducibi\-lity. Also, they are pairwise non-associate by Proposition \ref{samefiber}. 

Let $\phi$ be a key polynomial of minimal degree $n$.

By Proposition \ref{mu-irr}, every homogeneous prime element is associate either to $\hm(\phi)$ (which belongs to $H\pset$), or to $\psi(\xi)$ for some irreducible polynomial $\psi\in\kal[y]$, $\psi\ne y$. 

In the latter case, along the proof of Theorem \ref{Max} we saw the existence of $\chi\in\kpm$ such that $s(\chi)=0$ and $R(\chi)=\psi$. 

Therefore, $\psi(\xi)=R(\chi)(\xi)$ is  associate to $\hm(\chi)\in H\pset$, by Theorem \ref{Hmug}.

Finally, every homogeneous element in $\ggm$ is associate to a product of homogeneous prime elements, by Theorem \ref{Hmug} and Lemma \ref{prime0}.
\end{proof}

\section{Augmentation of valuations}\label{secAugm}
We keep dealing with a valuation $\mu$ on $\kx$ with $\kpm\ne\emptyset$.

There are different procedures to \emph{augment} this valuation, in order to obtain valuations $\mu'$ on $\kx$ such that
$$\mu(f)\le\mu'(f),\qquad \forall\, f\in\kx,$$
after embedding the value groups of $\mu$ and $\mu'$ in a common ordered group.

In this section, we show how to single out a key polynomial of $\mu'$ of minimal degree.

As an application, all results of this paper apply to determine the structure of $\gg_{\mu'}$ and the set $\op{KP}(\mu')/\!\!\sim_{\mu'}$ in terms of this key polynomial. 

\subsection{Ordinary augmentations}
\begin{definition}\label{muprima}
Take $\chi\in \kpm$. Let $\gm\hookrightarrow\g'$ be an order-preserving embedding of $\gm$ into another ordered group, and choose $\ga\in\g'$ such that $\mu(\chi)<\ga$. 

The \emph{augmented valuation} of $\mu$ with respect to these data is the mapping 
$$\mu':\kx\rightarrow \g' \cup \left\{\infty\right\}
$$
assigning to any $g\in\kx$, with canonical $\chi$-expansion $g=\sum_{0\le s}g_s\chi^s$, the value
$$
\mu'(g)=\Min\{\mu(g_s)+s\ga\mid 0\le s\}.
$$
We use the notation $\mu'=[\mu;\chi,\ga]$. Note that $\mu'(\chi)=\ga$.
\end{definition}

The following proposition collects several results of \cite[sec. 1.1]{Vaq}.

\begin{proposition}\label{extension} 
\mbox{\null}
\begin{enumerate}
\item The mapping $\mu'=[\mu;\chi,\ga]$ is a valuation on $\kx$ extending $v$, with value group $\g_{\mu'}=\gen{\g_{\deg(\chi)},\ga}$. 
\item For all $f\in\kx$, we have $\mu(f)\le\mu'(f)$. 

Equality holds if and only if $\chi\nmid_{\mu}f$ or $f=0$.
\item If $\chi\nmid_{\mu}f$, then $H_{\mu'}(f)$ is a unit in $\gg_{\mu'}$.
\item The polynomial $\chi$ is a key polynomial for $\mu'$.
\item The kernel of the canonical homomorphism 
$$\ggm\lra\gg_{\mu'},\qquad a+\pset_\al(\mu)\,\longmapsto\, a+\pset_\al(\mu'),\ \forall\,\al\in\gm,$$ is the principal ideal of $\ggm$ generated by $\hm(\chi)$.
\end{enumerate}
\end{proposition}

\begin{corollary}
The polynomial $\chi$ is a key polynomial for $\mu'$, of minimal degree.
\end{corollary}

\begin{proof}
For any polynomial $f\in\kx$ with $\deg(f)<\deg(\chi)$, we have $\chi\nmid_{\mu}f$ because $\chi$ is $\mu$-minimal. By item (3) of Proposition \ref{extension}, $H_{\mu'}(f)$ is a unit in $\gg_{\mu'}$, so that $f$ cannot be a key polynomial for $\mu'$.   
\end{proof}

\subsection{Limit augmentations}
Consider a totally ordered set $A$, not containing a maximal element.

A \emph{continuous MacLane chain} based on $\mu$, and parameterized by $A$, is a family $\left(\mua\right)_{\al\in A}$ of augmented valuations
$$
\mua=[\mu;\phi_\al,\ga_\al],\qquad \phi_\al\in\kpm,\quad \mu(\phi_\al)<\ga_\al\in\gm,
$$
for all $\al\in A$, satisfying the following conditions:
\begin{enumerate}
\item $\deg(\phi_\al)=d$ is independent of $\al\in A$.
\item The mapping $A\to \gm$,  $\al\mapsto \ga_\al$ is an order-preserving embedding of $A$ in $\gm$. 
\item For all $\al<\beta$ in $A$, we have
$$
\phi_\beta\in \op{KP}(\mua),\qquad \phi_\al\not\sim_{\mua}\phi_\beta,\qquad \mu_\beta=[\mua;\phi_\beta,\ga_\beta].
$$
\end{enumerate}

In Vaqui\'e's terminology, $\left(\mua\right)_{\al\in A}$ is a ``famille continue de valuations augment\'ees it\'er\'ees" \cite{Vaq}.

\begin{definition}
A polynomial $f\in\kx$ is \emph{$A$-stable} if there exists $\al_0\in A$ such that 
$$
\mu_{\al_0}(f)=\mua(f),\qquad \forall\,\al\ge \al_0.
$$

In this case, we denote by $\mu_A(f)$ this stable value.
\end{definition}

\begin{lemma}\label{nonstable}
 If $f\in\kx$ is not $A$-stable, then $\mua(f)<\mu_\beta(f)$ for all $\al<\beta$ in $A$.
\end{lemma}

\begin{proof}
Let us show that the equality $\mua(f)=\mu_{\beta}(f)$ for some $\al<\beta$ in $A$, implies that $f$ is $A$-stable.  

Since $\mu_{\beta}=[\mua;\phi_\beta,\ga_\beta]$, Proposition \ref{extension} shows that $\phi_\beta\nmid_{\mua}f$ and $H_{\mu_{\beta}}(f)$ is a unit in $\gg_{\mu_{\beta}}$.
Hence, for all $\delta\ge\beta$, the image of $H_{\mu_{\beta}}(f)$ in $\gg_{\mu_\delta}$ is a unit. 

Since $\mu_\delta=[\mu_{\beta};\phi_\delta,\ga_\delta]$, item (5) of Proposition \ref{extension} shows that $\phi_\delta\nmid_{\mu_{\beta}}f$. Hence, $\mu_{\beta}(f)=\mu_\delta(f)$, again by Proposition \ref{extension}. Thus, $f$ is $A$-stable. 
\end{proof}\bigskip

If all polynomials in $\kx$ are $A$-stable, then $\mu_A=\lim_{\al\in A}\mua$ has an obvious meaning, and $\mu_A$ is a valuation on $\kx$.

If there are polynomials which are not $A$-stable, there are still some particular situations in which $\left(\mua\right)_{\al\in A}$ converges to a valuation (or semivaluation) on $\kx$.

However, regardless of the fact that  $\left(\mua\right)_{\al\in A}$ converges or not, non-stable polynomials may be used to define \emph{limit augmented} valuations of this continuous MacLane chain.

Let us assume that not all polynomials in $\kx$ are $A$-stable. 

We take a monic $\phi\in\kx$ which is not $A$-stable, and has minimal degree among all polynomials having this property.

Since the product of $A$-stable polynomials is $A$-stable, $\phi$ is irreducible in $\kx$.

\begin{lemma}\label{minimalA}
Let $f\in\kx$ be a non-zero polynomial, with canonical $\phi$-expansion $f=\sum_{0\le s}f_s\phi^s$. Then, there exist an index $s_0$ 
and an element $\al_0\in A$ such that
$$
\mua(f)=\mua(f_{s_0}\phi^{s_0})<\mua(f_s\phi^s),\qquad \forall\,s\ne s_0,\ \forall\,\al\ge\al_0.
$$
\end{lemma}

\begin{proof}
Since all coefficients $f_s$ have degree less than $\deg(\phi)$, they are all $A$-stable. Let us take $\al_1$ sufficiently large so that $\mua(f_s)=\mu_A(f_s)$ for all $s\ge0$ and all $\al\ge\al_1$.

For every $\al\in A$, $\al\ge\al_1$, let
$$
\delta_\al=\mn\{\mua(f_s\phi^s)\mid 0\le s\},\qquad I_\al=\left\{s\mid \mua(f_s\phi^s)=\delta_\al\right\},\qquad s_\al=\mn(I_\al).
$$
For any index $s$ we have
\begin{equation}\label{s0}
\mua(f_{s_\al}\phi^{s_\al})=\mu_A(f_{s_\al})+s_\al\mua(\phi)\le \mua(f_s\phi^s)=\mu_A(f_s)+s\mua(\phi). 
\end{equation}

Since $\phi$ is not $A$-stable, Lemma \ref{nonstable} shows that $\mua(\phi)<\mu_\beta(\phi)$, for all $\beta>\al$, $\beta\in A$. 
Thus, if we replace $\mua$ with $\mu_\beta$ in (\ref{s0}), we get a strict inequality for all $s>s_\al$, because the left-hand side of (\ref{s0}) increases by $s_\al(\mu_\beta(\phi)-\mua(\phi))$, while the right-hand side increases by $s(\mu_\beta(\phi)-\mua(\phi))$.

Therefore, either $I_\beta=\{s_\al\}$, or $s_\beta=\mn(I_\beta)<s_\al$. 

Since $A$ contains no maximal element, we may consider a strictly increasing infinite sequence of values of $\beta\in A$. There must be an $\al_0\in A$ such that 
$$I_\al=\{s_0\},\qquad \forall\,\al\ge\al_0,
$$
because the set of indices $s$ is finite.
\end{proof}

\begin{definition}
For any $f\in\kx$, we say that $f$ is $A$-divisible by $\phi$, and we write $\phi\mid_A f$ if there exists $\al_0\in A$ such that $\phi\mid_{\mua}f$ for all $\al\ge\al_0$.
\end{definition}

\begin{lemma}\label{stableChar}
For any $f\in\kx$ with canonical $\phi$-expansion $f=\sum_{0\le s}f_s\phi^s$, the following conditions are equivalent:
\begin{enumerate}
\item $f$ is $A$-stable.
\item There exists $\al_0\in A$ such that
$$
\mua(f)=\mua(f_0)<\mua(f_s\phi^s),\quad \forall s>0, \ \forall\,\al\ge\al_0.
$$
\item $\phi\nmid_Af$.
\end{enumerate}
\end{lemma}

\begin{proof}
By Lemma \ref{minimalA}, there exist an index $s_0$ and an element $\al_0\in A$ such that 
$\mua(f_s)=\mu_A(f_s)$ for all $s$, and
$$\mua(f)=\mua\left(f_{s_0}\phi^{s_0}\right)=\mu_A(f_{s_0})+s_0\mua(\phi)<\mua\left(f_s\phi^s\right),\qquad \forall\,s\ne s_0.$$

By Lemma \ref{nonstable}, $\mua(\phi)$ grows strictly with $\al$. Hence, condition (1) is equivalent to $s_0=0$, which is in turn equivalent to condition (2).

On the other hand, Proposition \ref{minimal0} shows that $\phi$ is $\mua$-minimal for all $\al\ge\al_0$. Hence, again by Proposition \ref{minimal0}, the three following conditions are equivalent:
\begin{itemize}
\item $\phi\nmid_{\mua}f$ for all $\al\ge\al_0$. 
\item $\phi\nmid_{\mua}f$ for some $\al\ge\al_0$.
\item $s_0=0$.
\end{itemize}

Hence, conditions (2) and (3) are equivalent too.
\end{proof}

\begin{definition}\label{muprimalim}
Take $\phi\in \kx$ a monic polynomial with minimal degree among all polynomials which are not $A$-stable.

Let $\gm\hookrightarrow\g'$ be an order-preserving embedding of $\gm$ into another ordered group, and choose $\ga\in\g'$ such that $\mua(\phi)<\ga$ for all $\al\in A$. 

The \emph{limit augmented valuation} of the continuous MacLane chain $\left(\mua\right)_{\al\in A}$ with respect to these data is the mapping 
$$\mu':\kx\rightarrow \g' \cup \left\{\infty\right\}
$$
assigning to any $g\in\kx$, with $\phi$-expansion $g=\sum_{0\le s}g_s\phi^s$, the value
$$
\mu'(g)=\Min\{\mu_A(g_s)+s\ga\mid 0\le s\}.
$$
We use the notation $\mu'=[\mu_A;\phi,\ga]$. Note that $\mu'(\phi)=\ga$.
\end{definition}

The following proposition collects Propositions 1.22 and 1.23 of \cite{Vaq}.

\begin{proposition}\label{extensionlim} 
\mbox{\null}
\begin{enumerate}
\item The mapping $\mu'=[\mu_A;\phi,\ga]$ is a valuation on $\kx$ extending $v$. 
\item For all $f\in\kx$, we have $\mua(f)\le\mu'(f)$ for all $\al\in A$. 

\noindent The condition $\mua(f)<\mu'(f)$, for all $\al\in A$, is equivalent to $\phi\mid_Af$ and $f\ne0$.
\end{enumerate}
\end{proposition}

\begin{lemma}\label{abcd}
For $a,b\in\kx_{\deg(\phi)}$, let the $\phi$-expansion of $ab$ be
$$ab=c+d\phi,\qquad c,d\in\kx_{\deg(\phi)}.$$ Then, $ab\sim_{\mu'} c$.
\end{lemma}

\begin{proof}
Since $\deg(a),\deg(b)<\deg(\phi)$, the polynomials $a$ and $b$ are $A$-stable. In particular, $ab$ is $A$-stable. 
By Lemma \ref{stableChar}, there exists $\al_0\in A$ such that
$$
\mu_A(c)=\mua(c),\quad \mbox{ and }\quad \mua(ab)=\mua(c)<\mua(d\phi), \qquad \forall\,\al\ge\al_0\ \mbox{ in }A.
$$

Hence, $\mu_A(ab)=\mu_A(c)<\mu_A(d)+\mua(\phi)<\mu_A(d)+\ga$. By the definition of $\mu'$, we conclude that $\mu'(ab)=\mu'(c)<\mu'(d\phi)$.
\end{proof}

\begin{corollary}\label{smallunitslim}
For all $a\in\kx_{\deg(\phi)}$, $H_{\mu'}(a)$ is a unit in $\gg_{\mu'}$.
\end{corollary}

\begin{proof}
Since $\phi$ is irreducible in $\kx$ and $\deg(a)<\deg(\phi)$, these two polynomials are coprime. Hence, there is a B\'ezout identity:
$$
ac+d\phi=1,\qquad \deg(c)<\deg(\phi),\ \deg(d)<\deg(a)<\deg(\phi).
$$
By Lemma \ref{abcd}, $ac\sim_{\mu'}1$.
\end{proof}

\begin{corollary}
The polynomial $\phi$ is a key polynomial for $\mu'$, of minimal degree.
\end{corollary}

\begin{proof}
By the very definition of $\mu'$ we have $\mu'(f)=\mn\{\mu'\left(f_s\phi^s\right)\mid 0\le s\}$ for any $f=\sum_{0\le s}f_s\phi^s\in\kx$. By Proposition \ref{minimal0}, $\phi$ is $\mu'$-minimal.\e

Let us show that $\phi$ is $\mu'$-irreducible. Suppose $\phi\nmid_{\mu'}f$,  $\phi\nmid_{\mu'}g$ for $f,g\in\kx$. Let $f_0,g_0$ be the $0$-th coefficients of the $\phi$-expansion of $f$ and $g$, respectively. 

Since $\phi$ is $\mu'$-minimal, Proposition \ref{minimal0} shows that $\mu'(f)=\mu'(f_0)$ and $\mu'(g)=\mu'(g_0)$.

Consider the $\phi$-expansion 
$$
f_0g_0=c+d\phi,\qquad \deg(c),\deg(d)<\deg(\phi).
$$
By Lemma \ref{abcd}, $\mu'(f_0g_0)=\mu'(c)$, so that
$$
\mu'(fg)=\mu'(f_0g_0)=\mu'(c).
$$
Since the polynomial $c$ is the $0$-th coefficient of the $\phi$-expansion of $fg$, Proposition \ref{minimal0} shows that $\phi\nmid_{\mu'}fg$. Hence, $\phi$ is $\mu$-irreducible. 

This shows that $\phi$ is a key polynomial for $\mu'$.\e

Finally, by Corollary \ref{smallunitslim}, any polynomial in $\kx$ of degree smaller than $\deg(\phi)$ cannot be a key polynomial for $\mu'$, because it is a $\mu'$-unit.
\end{proof}

\end{document}